\documentclass[12pt, reqno]{amsart}
\usepackage{amssymb, hyperref}
\usepackage[all,2cell]{xy}
	\usepackage{amssymb, hyperref}
	\usepackage[all,2cell]{xy}
	\usepackage{hyperref}
	\usepackage{array}
	\usepackage[capitalise]{cleveref}
	\usepackage{geometry}
	\numberwithin{equation}{section}
	\newtheorem{theorem}{Theorem}[section]
	\newtheorem{corollary}[theorem]{Corollary}
	\newtheorem{lemma}[theorem]{Lemma}
	
	\newtheorem{proposition}[theorem]{Proposition}

	\theoremstyle{remark}
	\newtheorem{rem}[theorem]{Remark}
	\newtheorem{exam}[theorem]{Example}
	\newtheorem{definition}[theorem]{Definition}

\begin{document}
		\title[$G$-Bott-Samelson-Demazure-Hansen variety]{Line bundles on $G$-Bott-Samelson-Demazure-Hansen varieties}
		
		\author[S. Bhaumik]{Saurav Bhaumik}
		
		\address{Department of Mathematics\\ IIT Bombay\\ Mumbai 400076}
		
		\email{saurav@math.iitb.ac.in}
		
		\author[P. Saha]{Pinakinath Saha}
		
		\address{Department of Mathematics\\ IIT Bombay\\ Mumbai 400076}
		
		\email{psaha@math.iitb.ac.in}
		
		
		
		\begin{abstract}
			Let $G$ be a semi-simple simply connected algebraic group over an algebraically closed field $k$ of arbitrary characteristic. Let $B$ be a Borel subgroup of $G$ containing a maximal torus $T$ of $G.$ Let $W$ be the Weyl group of $G$ with respect to $T$. For an arbitrary sequence $w=(s_{i_{1}},s_{i_{2}},\ldots, s_{i_{r}})$ of simple reflections in $W,$ let $Z_{w}$ be the Bott-Samelson-Demazure-Hansen variety (BSDH-variety for short) corresponding to $w.$  Let $\widetilde{Z_{w}}:=G\times^{B}Z_{w}$ denote the fibre bundle over $G/B$ with the fibre over $B/B$ is $Z_{w}.$ In this article, we give necessary and sufficient conditions for the varieties $Z_{w}$ and $\widetilde{Z_{w}}$ to be Fano (weak-Fano). We show that a line bundle on $Z_{w}$ is globally generated if and only if it is nef. We show that Picard group $\text{Pic}(\widetilde{Z_{w}})$ is free abelian and we construct a $\mathcal{O}(1)$-basis. We characterize the nef, globally generated, ample and very ample line bundles on $\widetilde{Z_{w}}$ in terms of the $\mathcal{O}(1)$-basis. 	
		\end{abstract}	
		\maketitle 
		
		\tableofcontents
		\section{Introduction}
		Let $G$ be a semi-simple simply connected algebraic group over an algebraically closed filed $k$ of arbitrary characteristic. Let $B$ be a Borel subgroup of $G$ containing a maximal torus $T$ of $G.$ Let $W\,=\,N_{G}(T)/T$ denote the Weyl group of $G$ with respect to $T,$ where $N_{G}(T)$ denotes the normalizer of $T$ in $G.$ Let
		$w\,=\,(s_{i_{1}},\,\ldots ,\,s_{i_{r}})$ be a sequence (not necessarily reduced) of simple reflections in $W.$ Let $Z_{w}$ be the Bott-Samelson-Demazure-Hansen variety (BSDH-variety for short) corresponding to $w.$ For any $w\,\in\, W,$ let $X(w):=\overline{BwB}/B$ denote the Schubert variety in $G/B$ corresponding to $w.$ For a sequence of simple reflections $w=(s_{i_{1}},s_{i_{2}},\ldots, s_{i_{r}})$ such that $w=s_{i_{1}}\cdots s_{i_{r}}\in W$ is a reduced expression, $Z_{w}$ is a natural desingularization of $X(w)$. It was first introduced by Bott and Samelson in a differential geometric and topological context (cf. \cite{BS}). Demazure in \cite{De} and Hansen in \cite{Han}
		independently adapted the construction in algebro-geometric situation, which explains the name.
		
		There is natural action of $B$ on $Z_{w}$ given by the multiplication on the left. Let $\widetilde{Z_{w}}:=G\times^{B} Z_{w}$ be the fibre bundle associated to the principal $B$-bundle $G\longrightarrow G/B.$ Then $\widetilde{Z_{w}}$ is a smooth projective variety. For a sequence of simple reflections $w=(s_{i_{1}},s_{i_{2}},\ldots, s_{i_{r}})$ such that $w=s_{i_{1}}\cdots s_{i_{r}}\in W$ is a reduced expression, $\widetilde{Z_{w}}$ is a natural desingularization of $G\times^{B}X(w).$ We call it $G$-Bott-Samelson-Demazure-Hansen variety ($G$-BSDH-variety for short) corresponding to $w.$
		In \cite{MR88} Mehta and Ramanathan calculated the description of the anti-canonical line bundle of $\widetilde{Z_{w}}$ and from this description of the anti-canonical line bundle  of $\widetilde{Z_{w}},$ they proved that $\widetilde{Z_{w}}$ is Frobenius-split.

		In \cite{LLM}, Lakshmibai, Littelmann and Magyar studied the Standard Monomial Theory for globally generated (respectively, ample) line bundles on $Z_{w}.$ As a consequence, they proved the cohomology vanishing theorems for globally generated (and also ample) line bundles. 
		In \cite{LT}, Lauritzen and Thompsen characterized the globally generated, ample and very ample line bundles on $Z_{w}.$ 
		
		A smooth projective variety $X$ over $k$ is called {\it Fano} (repectively, {\it weak-Fano}) if its anti-canonical line bundle $\omega_{X}^{-1}$ is ample (respectively, nef and big).
		
		In this article, we prove the following results.	
		\begin{theorem}\label{thm4.16}
			Let $w=(s_{i_{1}},s_{i_{2}},\ldots, s_{i_{r}})$ be an arbitrary sequence of simple reflections. Then a line bundle $\mathcal{L}$ on $Z_{w}$ is  globally generated if and only if it is nef.
		\end{theorem}
		
		To proceed further we introduce the following notation for be an arbitrary sequence $w=(s_{i_{1}},s_{i_{2}},\ldots, s_{i_{r}})$ of simple reflections. For  $1\le j\le r,$  let 
		\begin{equation*}
			m_{jj}=\langle \alpha_{i_{j}},\alpha_{i_{j}}^{\vee}\rangle =2
		\end{equation*}
		and 
		\begin{equation*}
			m_{jk}=\langle \alpha_{i_{j}}-\sum_{l=k+1}^{j}m_{jl}\varpi_{i_{l}},\alpha_{i_{k}}^{\vee}\rangle 
		\end{equation*} for all $1\le k\le j-1.$ 
		Note that using the previous recursive formula, $m_{jk}$ (~$1\le k\le j\le r$~) are well determined. 
		Then we prove 
		\begin{theorem}\label{cor4.18}
		$Z_{w}$ is  Fano \rm{(weak-Fano)} if and only if  $\sum_{l=j}^{r}m_{lj} >0$ ($\sum_{l=j}^{r}m_{lj} \ge 0$) for all $1\le j\le r-1.$  
		\end{theorem}
		
		B. N. Chary \cite{Cha} characterized Fano (weak-Fano) BSDH-varieties for reduced expressions. Our \cref{cor4.18} provides a generalization  and different characterization of Fano (weak-Fano) BSDH-varieties.

		In \cref{sec4} we construct a $\mathcal{O}(1)$-basis for the line bundles on $\widetilde{Z_{w}}.$ The variety $\widetilde{Z_{w}}$ admits a sequence of $\mathbb{P}^{1}$-fibrations to $G/B.$ The $\mathcal{O}(1)$-bundles for the successive $\mathbb{P}^{1}$-fibrations together with the generators of the ample cone of the Picard group ${\rm Pic}(G/B)$ form the generators of the ample cone of ${\rm Pic}(\widetilde{Z_{w}}).$ They form the $\mathcal{O}(1)$-basis. 
		
		\begin{lemma} \label{lem4.1}
		Any line bundle on $\widetilde{Z_{w}}$ is of the form  $$\mathcal{O}_{\widetilde{Z_{w}}}(m_{1},\ldots, m_{r})\otimes \widetilde{\pi_{w[r]}}^{*}\mathcal{L}(\lambda)$$ for a unique  $(m_{1},\ldots, m_{r})\in \mathbb{Z}^{r}$ and a unique $\lambda\in X(B).$
		\end{lemma}

		\begin{theorem}\label{thm3.1} Let $(m_{1},\ldots,m_{r})\in \mathbb{Z}^{r}$ and $\lambda\in X(B).$ Then the line bundle $$\mathcal{L}=\mathcal{O}_{\widetilde{Z_{w}}}(m_{1},\ldots, m_{r})\otimes \widetilde{\pi_{w[r]}}^{*}\mathcal{L}(\lambda)$$ on $\widetilde{Z_{w}}$ is very ample if and only if $m_{1},\ldots, m_{r}>0$ and $\lambda$ is regular dominant. 
		\end{theorem}
		As a consequence of \cref{thm3.1} we prove that any ample line bundle on $\widetilde{Z_{w}}$ is very ample (cf. \cref{cor4.3}). Also, we show that a line bundle $\mathcal{L}= \mathcal{O}_{\widetilde{Z_{w}}}(m_{1},\ldots ,m_{r})\otimes\widetilde{\pi_{w[r]}}^{*}\mathcal{L}(\lambda)$ on $\widetilde{Z_{w}}$ is globally generated if and only if $m_{1},\dots,m_{r}\ge 0$ and $\lambda$ is dominant (cf. \cref{cor3.3}). 
		
		\begin{theorem}\label{thm3.5}
			A line bundle $\mathcal{L}$ on $\widetilde{Z_{w}}$ is  globally generated if and only if $\mathcal{L}$ is nef.
		\end{theorem}

		Let $F$ be a smooth projective $B$-variety. Let $D'$ be a $B$-invariant effective divisor on $F.$ Then consider the variety $X= G\times^{B} F$ and let $\pi: X\longrightarrow G/B$ be the natural projection map defined by $\pi([g,f])=gB$ for $g\in G$ and $f\in F.$ Let $\widetilde{D'}:= G\times^{B} D'$ denote the $G$-invariant effective divisor on $X.$ 
		
		Let $p: G\longrightarrow G/B$ be the natural projection map.
		Let $\widetilde{D''}:=\pi^{-1}(\partial G/B)=p^{-1}(\partial G/B)\times^{B} F,$ where $\partial G/B$ denotes the boundary divisor of $G/B.$ Let $\widetilde{D}:= \widetilde{D'}+\widetilde{D''}.$ 
		
		\begin{lemma} \label{lem4.8}
			If $F-{\rm Supp}(D')$ is affine, then the effective divisor $\widetilde{D}$ is big on $X.$
		\end{lemma}
	
 Using \cref{lem4.8} and the description of the anti-canonical line bundle $\omega_{\widetilde{Z_{w}}}^{-1}$ of $\widetilde{Z_{w}}$ due to Mehta-Ramanathan (cf. \cite[Proposition 2]{MR88}) we prove that  $\omega_{\widetilde{Z_{w}}}^{-1}$ is big. Using \cref{thm3.5} and \cref{lem4.8} we prove that $\widetilde{Z_{w}}$ is weak-Fano if and only if  $\omega_{\widetilde{Z_{w}}}^{-1}$ is globally generated. We also prove that $\omega_{\widetilde{Z_{w}}}^{-1}$  is globally generated if and only if  it is nef (cf. \cref{cor4.10}).
 
 For a fix integer $1\le j\le r,$ let $\lambda_{j}=\alpha_{i_{j}}-\sum_{k=1}^{j}m_{jk}\varpi_{i_{j}}.$ 
 \begin{theorem}\label{thm5.7}
  $\widetilde{Z_{w}}$ is weak-Fano (respectively, Fano) if and only if 
 		\begin{itemize}
 			\item [(i)] $\sum_{l=j}^{r}m_{lj}\ge 0 ~(\text{respectively,} \sum_{l=j}^{r}m_{lj} >0)$ for all $1\le j\le r-1.$
 			
 			\item [(ii)] $2\rho+\sum_{j=1}^{r}\lambda_{j}$ is dominant (respectively, regular dominant).
 		\end{itemize} 
 \end{theorem}

Even if  $Z_{w}$ is Fano (respectively, weak-Fano), $\widetilde{Z_{w}}$ may not be Fano (respectively, weak-Fano) (cf. \cref{exam5.9}).

A vector bundle $\mathcal{E}$ on a projective variety $X$ is said to be  {\it nef}
(respectively,  {\it ample}) if $\mathcal{O}_{\mathbb{P}(\mathcal{E})}(1)$ is a {\it nef} (respectively, {\it ample}) line bundle on the projective
bundle $\mathbb{P}(\mathcal{E})$ (parameterizing the quotient line bundles of $\mathcal{E}$) over $X.$ 

Let $\mathcal{E}$ be a vector bundle on $X.$ If $\mathcal{E}$ is ample, then its restriction to every curve (closed irreducible subvariety of dimension one) is also ample.  The converse is not true in general. Mumford gave an example of a non-ample line bundle on a surface which intersects
every curve positively (cf. \cite[Example 10.6]{Har1} or \cite[Example 1.5.2]{L}). In some special situations the converse does hold, for example when $\mathcal{E}$ is a torus-equivariant vector bundle on a toric variety (cf. \cite{HMP}), a generalized
flag variety (cf. \cite{BHN}), a  Bott-Samelson-Demazure-Hansen (BSDH) variety or a wonderful compactification of a symmetric variety of minimal rank (cf. \cite{BHK}).

In this article we get a similar result for  $G$-BSDH varieties.

\begin{theorem}\label{thm6.3}
Let $w=(s_{i_{1}},s_{i_{2}},\ldots, s_{i_{r}})$ be a reduced sequence of simple reflections. Then a $T$-equivariant vector bundle $\mathcal{E}$ on $\widetilde{Z_{w}}$ is nef (respectively, ample) if and only if the restriction $\mathcal{E}|_{C}$ of $\mathcal{E}$ to every $T$-invariant curve $C$ on $\widetilde{Z_{w}}$ is nef (respectively, ample).
\end{theorem}

The layout of our article is as follows: In \cref{sec2} we introduce some notations and preliminaries. In \cref{sec3} we study BSDH varieties and prove \cref{thm4.16} and \cref{cor4.18}. In \cref{sec4} we study $G$-BSDH varieties and prove \cref{lem4.1}, \cref{thm3.1}, \cref{thm3.5} and \cref{lem4.8}. In \cref{sec5} we prove \cref{thm5.7}. In \cref{sec6} we prove \cref{thm6.3}.

		\section{Notations and Preliminaries}\label{sec2}
		
		In this section we recall some notations and preliminaries. For details on 
		algebraic groups, Lie algebras and Bott-Samelson-Demazure-Hansen varieties \cite{BK05}, \cite{Hum1}, \cite{Hum2}, \cite{Jan} and \cite{Spr} 
		are referred.
		
		We denote the set of roots of $G$ with respect to $T$ by $R$. Let $B^{+}$ be a Borel subgroup of $G$ 
		containing $T$. The Borel subgroup of $G$ opposite to $B^{+}$ determined by $T$ is denoted by
		$B$, in other words, $B\,=\,n_{0}B^{+}n_{0}^{-1}$, where $n_{0}$ 
		is a representative in $N_{G}(T)$ of the longest element $w_{0}$ of $W$. Let $R^{+}\,\subset\, R$ be the set of positive roots of $G$ with 
		respect to the Borel subgroup $B^{+}$.
		The set of positive roots of $B$ is equal 
		to the set $R^{-} \,:=\, -R^{+}$ of negative roots. Let $S \,=\, \{\alpha_1,\,\ldots,\,\alpha_n\}$ denote the set of simple roots in $R^{+},$
		where $n$ is the rank of $G.$ The simple reflection in $W$ corresponding to $\alpha_i$ is denoted by $s_{i}.$
		
		Let $X(T)$ (respectively, $Y(T)$) denote the group of the characters (respectively, co-characters) of $T.$ Let $\langle \cdot, \cdot \rangle: X(T)\times Y(T)\longrightarrow \mathbb{Z}$ be the natural bilinear perfect pairing. Let ${X(T)}_{\mathbb{R}}:=X(T)\otimes _{\mathbb{Z}}\mathbb{R}$ and ${Y(T)}_{\mathbb{R}}:=Y(T)\otimes_{\mathbb{Z}}\mathbb{R}.$ We also denote the induced bilinear pairing  ${X(T)}_{\mathbb{R}}\times {Y(T)}_{\mathbb{R}}\longrightarrow \mathbb{R}$ by $\langle \cdot, \cdot\rangle.$ Let $(\cdot, \cdot)$ be a positive definite $W$-invariant symmetric bilinear form on ${X(T)}_{\mathbb{R}}.$ For a simple root $\alpha,$ the simple reflection $s_{\alpha}$ satisfies $s_{\alpha}(\lambda)=\lambda-\frac{2(\lambda,\alpha)}{(\alpha,\alpha)}\alpha$  for all $\lambda\in {X(T)}_{\mathbb{R}}.$ 
		For a root $\alpha\in R^{+},$ we define the co-root $\alpha^{\vee}\in {Y(T)}_{\mathbb{R}}$ by $\langle \lambda, \alpha^{\vee}\rangle:=\frac{2(\lambda,\alpha)}{(\alpha,\alpha)}.$ Let $X(T)^+$ denote the set of dominant characters of $T$ with respect to $B^{+}$ (i.e., $\lambda$ such that $\langle \lambda, \alpha_{i}^{\vee}\rangle\in \mathbb{Z}_{\ge 0}$ for all $1\le i\le n$). For any simple root $\alpha_{i}$, we denote the fundamental weight
		corresponding to $\alpha_{i}$ by $\varpi_{i},$ i.e., $\langle \varpi_{i}, \alpha_{j}^{\vee}\rangle=\delta_{ij},$ where $\delta_{ij}$ denotes the Kronecker delta. Let $\rho$ be the half sum of all 
		positive roots of $G.$

		\subsection{BSDH-variety} 
		
		We recall the definition of BSDH-vartiety and some of its properties following Lauritzen and Thomsen \cite{LT}. 
		
		For a simple root $\alpha \,\in\, S,$ let $n_{\alpha}\,\in\, N_{G}(T)$ be a representative of $s_{\alpha}.$ The
		unique minimal parabolic subgroup of $G$ containing $B$ and $n_{\alpha}$ is denoted by $P_{\alpha}.$
		
		We recall that the Bott-Samelson-Demazure-Hansen variety (BSDH-variety for short) corresponding to a sequence $w=\,(s_{i_{1}},\,s_{i_{2}},\,\ldots,\,s_{i_{r}})$
		of simple reflections in $W$ is defined by  $Z_{w}=P_{w}/B^{r},$ where $P_{w}=P_{\alpha_{i_{1}}}\times\cdots \times P_{\alpha_{i_{r}}}$ and $B^{r}$ acts from the right on $P_{w}$ as $(p_{1} ,\ldots , p_{r} )(b_1,\ldots ,b_{r} )=(p_1b_1, b_{1}^{-1}p_{2}b_{2}, b_{2}^{-1}p_{3}b_{3} , \ldots, b_{r-1}^{-1}p_{r}b_{r}).$
		
		We may also write 
		$Z_{w}\,=\,
		P_{\alpha_{i_{1}}}\times^{B} P_{\alpha_{i_{2}}}\times^{B}\cdots \times^{B} 
		P_{\alpha_{i_{r}}}/B,$
		where $X\times^{B} Y$ is the quotient $(X\times Y )/B$ with $B$ acting as $(x, y)b=(xb, b^{-1}y)$
		for all $x\in X, y\in Y$ and $b\in B.$ The equivalence class of $(x,y)$ is denoted by $[x,y].$ There is a natural action of $B$ on $Z_{w}$ given by the multiplication on the left, i.e., $b\cdot [p_{1},\ldots, p_{r}]=[bp_{1},\ldots, p_{r}]$ for $b\in B$ and $[p_{1},\ldots, p_{r}]\in Z_{w}.$
		
		For $1\le j\le r,$ let $w[j]$ denote the truncated sequence
		$(s_{i_{1}},\ldots, s_{i_{r-j}}).$ By convention $Z_{w[r]}$ denotes a $1$-point space. Then  $Z_{w}$ comes as the sequence
		$Z_{w} \to Z_{w[1]}\to \cdots\to  Z_{w[r-2]}=P_{\alpha_{i_{1}}}\times^{B} P_{\alpha_{i_{2}}}/B \to Z_{w[r-1]}=P_{\alpha_{i_{1}}}/B$
		of successive $\mathbb{P}^{1}$-fibrations. In general, we let $\pi_{w[j]}$ denote the natural morphism $Z_{w}\to  Z_{w[j]}$ in the sequence of $\mathbb{P}^{1}$-bundles above.
		
		For a sequence $w=\,(s_{i_{1}},\,s_{i_{2}},\,\ldots,\,s_{i_{r}})$
		of simple reflections, we denote the element $\,s_{i_{1}}s_{i_{2}}\cdots s_{i_{r}}\in W$ also by $w.$ We say a sequence $w=(s_{i_{1}},\,s_{i_{2}},\,\ldots,\,s_{i_{r}})$
		of simple reflections is reduced if $w=s_{i_{1}}s_{i_{2}}\cdots s_{i_{r}}$ is a reduced expression in $W$ with respect to the set $\{s_{i}| 1\le i\le n\}$ of simple reflections. 
		
		We note that for any sequence $w=(s_{i_{1}},\,s_{i_{2}},\,\ldots,\,s_{i_{r}})$ of simple reflections, $Z_{w}$ is a smooth projective
		variety. Define $w[0]:=w.$ Then for each integer $0\le j\le r,$ $Z_{w[j]}$  comes with a $B$-equivariant morphism
		\begin{equation}\label{e1}
			\varphi_{w[j]}\,:\, Z_{w[j]}\,\longrightarrow\, G/B
		\end{equation}
		defined by $[p_{1},\,\cdots ,\, p_{j}]\, \longmapsto\, p_{1}\cdots p_{j}B.$ In particular, when $w$ is a reduced sequence of simple reflections, the natural map
		\begin{equation*}
			\varphi_{w}=\varphi_{w[0]}\,:\, Z_{w}\,\longrightarrow\, G/B
		\end{equation*} induces a
		birational surjective morphism from $Z_{w}$ to $X(w).$

		\subsection{Equivariant bundle on $G/B$ and $Z_{w}$}
		
		Let $\mathcal{O}_{Z_{w}}(V)$ denote the locally free sheaf of sections of the associated vector bundle $P_{w} \times^{B^{r}} V$ on $Z_{w},$
		where $V$ is a finite dimensional $B^{r}$-representation. We view a $B$-representation $V$ as a $B^{r}$-representation by letting $B^{r}$ act on $V$ as $(b_1, b_{2},\ldots,b_r)\cdot v= b_r\cdot v,$ where $v\in V, b_{i}\in B.$ With this convention we get for a $B$-character $\lambda\in X(B)$ that $\mathcal{O}_{Z_{w}}(\lambda)=\mathcal{O}_{Z_{w}}(0,\cdots, 0, \lambda).$
		
		For a $B$-module $V,$ let $\mathcal{L}(V)$ denote the locally free sheaf on $G/B$ corresponding to the vector bundle  $G\times^{B} V\longrightarrow G/B,$ where the action of $B$ is given by $(g, v)\cdot b=(gb, b^{-1}v)$ for $b\in B, g\in G$ and $v\in V.$  When $V$ has rank $1,$ associated to
		a $B$-character $\lambda$, the sheaf $\mathcal{L}(\lambda)= \mathcal{L}(V)$ is a line bundle. It is well known that the map $X(B)\longrightarrow {\rm Pic}(G/B):$ $\lambda\mapsto \mathcal{L}(\lambda)$ is a bijection. Moreover,   $\mathcal{L}(\lambda)$ is globally generated (respectively, ample) exactly when $\lambda$ is dominant (respectively, regular) i.e., $\langle \lambda, \alpha_{i}^\vee\rangle\ge 0$ (respectively, $\langle \lambda,\alpha_{i}^{\vee} \rangle>0$) for all $1\le i\le n$. The pull back of $\mathcal{L}(V)$ to $Z_{w}$ under the morphism $\varphi_{w}: Z_{w}\longrightarrow G/B$ satisfies $\varphi_{w} ^{*}\mathcal{L}(V)=\mathcal{O}_{Z_{w}}(V).$ In particular,  for $\lambda\in X(B),$  $\varphi_{w}^{*}\mathcal{L}(\lambda)=\mathcal{O}_{Z_{w}}(\lambda).$

		\subsection{Picard group of BSDH-variety}
		For any $1\,\le\, j\,\le\, r,$ let $w(j)=(s_{i_{1}} ,\ldots,\widehat{s_{i_{j}}},\ldots, s_{i_{r}}).$ The natural embedding $P_{\alpha_{i_{1}}}\times \cdots \times P_{\alpha_{i_{j-1}}}\times B\times P_{\alpha_{i_{j+1}}}\times \cdots \times P_{\alpha_{i_{r}}}\hookrightarrow P_{\alpha_{i_{1}}}\times \cdots \times P_{\alpha_{i_{j}}}\times \cdots \times P_{\alpha_{i_{r}}}$ induces a closed embedding $\sigma_{w,j}: Z_{w(j)}\to Z_{w},$ which makes $Z_{w(j)}$ into a $B$-invariant divisor of $Z_{w}.$ The divisor $\partial Z_{w}=\bigcup_{j=1}^r Z_{w(j)}$ is a normal crossing divisor in $Z_{w}.$
		
		Let $V_{\alpha_{i_{r}}}=H^{0}(P_{\alpha_{i_{r}}} /B, \mathcal{L}(\varpi_{i_{r}}))$ denote the natural two dimensional $B$-module. Then we have the following exact sequence of $B$-modules,
		\begin{equation*}
			0\to L_{s_{i_{r}}(\varpi_{i_{r}})}\to V_{\alpha_{i_{r}}}\to L_{\varpi_{i_{r}}}\to 0,
		\end{equation*}	
		where $L_{\lambda}$ denotes the one-dimensional representation of $B$ with character $\lambda.$
		This gives an exact sequence 
		\begin{equation*}
			0\to \mathcal{L}(s_{i_{r}}(\varpi_{i_{r}}))\to \mathcal{L}(V_{\alpha_{i_{r}}})\to  \mathcal{L}(\varpi_{i_{r}}) \to 0
		\end{equation*}
	of vector bundles on $G/B.$
		Pulling back by $\varphi_{w[1]}$, this induces a sequence 
		\begin{equation*}
			0\to \mathcal{O}_{Z_{w[1]}}(s_{i_{r}(\varpi_{i_{r}})})\to \mathcal{O}_{Z_{w[1]}}(V_{\alpha_{i_{r}}})\to  \mathcal{O}_{Z_{w[1]}}(\varpi_{i_{r}}) \to 0
		\end{equation*}
	of vector bundles on $Z_{w[1]}.$
		Then $Z_{w}$ is identified with the projective bundle $\mathbb{P}(\mathcal{O}_{Z_{w[1]}}(V_{\alpha_{i_{r}}})).$ Thus the corresponding universal quotient bundle is $\mathcal{O}_{Z_{w[1]}}(\varpi_{i_{r}}).$  So we define $\mathcal{O}_{Z_{w}}(1):=\mathcal{O}_{Z_{w}}(\varpi_{i_{r}}),$ and write $\mathcal{O}_{Z_{w}}(m):= \mathcal{O}_{Z_{w}}(1)^{\otimes m}.$ Note that $\mathcal{O}_{Z_{w}}(1)$ is a line bundle on $Z_{w}$ with degree one along the fibers of $\pi_{w[1]}.$
		
		The map $\pi_{w[1]}$ is a $P_{\alpha_{i_{r}}}/B\,\simeq\, \mathbb{P}^1$-fibration. So, for any line bundle $\mathcal{L}$ on $Z_{w},$ there is a unique line bundle $\mathcal{M}$ on $Z_{w[1]},$ and $m\in \mathbb{Z}$ such that $\mathcal{L}=\mathcal{O}_{Z_{w}}(m)\otimes \pi_{w[1]}^{*}\mathcal{M},$ where 	$\deg(\mathcal{L}_{y})=m$ for some (hence for all) $y \in Z_{w[1]}.$  Thus, we have
		${\rm Pic}(Z_{w})\,\simeq\,\mathbb{Z}\mathcal{O}_{Z_{w}}(1)\oplus {\rm Pic}(Z_{w[1]}).$ Moreover, by using induction on the number of $\mathbb{P}^{1}$-fibrations, the Picard group ${\rm Pic}(Z_{w})$ is a free abelian group of rank $r$ and the line bundles
		$\mathcal{O}_{Z_{w[j]}}(1)$ for $1\le j\le r$ forms a basis of ${\rm Pic}(Z_{w}),$ where we write $\mathcal{O}_{Z_{w[j]}}(1)$ instead of the pull back $\pi_{w[j]}^{*}\mathcal{O}_{Z_{w[j]}}(1).$ This basis is usually called the $\mathcal{O}(1)$-basis of $Z_{w}.$ For $(m_{1},m_{2},\ldots,m_{r})\in \mathbb{Z}^{r},$ we denote the line bundle $\mathcal{O}_{Z_{w}}(m_{r})\otimes \mathcal{O}_{Z_{w[1]}}(m_{r-1})\otimes \cdots \otimes \mathcal{O}_{Z_{w[r-1]}}(m_{1})$  simply by $\mathcal{O}_{Z_{w}}(m_{1},\ldots, m_{r-1},m_{r}).$

		\subsection{$G$-BSDH-variety}
		
		Let $w=(s_{i_{1}},\ldots, s_{i_{r-1}}, s_{i_{r}})$ be a sequence of simple reflections (not necessarily reduced) and $Z_{w}$ be as defined above. Note that there is a natural action of $B$ on $Z_{w}$ given by the multiplication on the left. Recall that for each integer $0\le j\le r,$ $Z_{w[j]}$  comes with a $B$-equivariant morphism
		\begin{equation*}
			\varphi_{w[j]}\,:\, Z_{w[j]}\,\longrightarrow\, G/B
		\end{equation*} (see \eqref{e1}). For each integer $0\le j\le r,$ consider the variety $\widetilde{Z_{w[j]}}:= G\times^{B} Z_{w[j]}.$ There is natural action of $G$ on $\widetilde{Z_{w[j]}}$ given by the multiplication of the left.
	
	Then for each integer $0\le j\le r,$ $\varphi_{w[j]}$ induces a natural $G$-equivariant map
		\begin{equation}\label{eq2.2}
			\widetilde{\varphi_{w[j]}}: \widetilde{Z_{w[j]}} \longrightarrow G/B\times G/B
		\end{equation}  given by $$[g,x]\mapsto [gB, g\varphi_{w[j]}(x)]$$
where $g\in G,$ $x\in Z_{w[j]}$ and the action of $G$ on $G/B \times G/B$ is given by the diagonal action.		
		
	    $\widetilde{Z_{w}}$ comes as the sequence
		$$\widetilde{Z_{w}}\to \widetilde{Z_{w[1]}}\to \cdots \to \widetilde{Z_{w[r-2]}} \to \widetilde{Z_{w[r-1]}} =G\times^{B}P_{\alpha_{i_{1}}}/B\to \widetilde{Z_{w[r]}}=G/B$$
		of successive $\mathbb{P}^1$-fibrations.
		For $1\le j\le r,$  $\pi_{w[j]}$ induces a natural $G$-equivariant morphism $\widetilde{\pi_{w[j]}}: \widetilde{Z_{w}}\longrightarrow \widetilde{Z_{w[j]}}.$

		\section{Line bundles on BSDH-varieties}\label{sec3}
		In this section we prove that every nef line bundle on $Z_{w}$ is globally generated. Then we give necessary and sufficient conditions for $Z_{w}$ to be Fano (weak-Fano). 
		
		Let $w=(s_{i_{1}},s_{i_{2}},\ldots, s_{i_{r}})$ be an arbitrary sequence of simple reflections and $Z_{w}$ be the BSDH-variety  corresponding to $w.$

		\begin{proof}[{Proof of \cref{thm4.16}}]
			Assume that $\mathcal{L}$ is globally generated. Then by \cite[Example 1.4.5., p.52]{L} it follows that $\mathcal{L}$ is nef. 
			
			Conversely, suppose that $\mathcal{L}$ is nef. We prove the converse by induction on $r,$ which is the dimension of $Z_{w}.$
			Recall that by \cite[Section 3.1, p.464]{LT}  $\mathcal{L}$ is of the form $$\mathcal{O}_{Z_{w}}(m_{1},m_{2},m_{3},\ldots,m_{r})$$ for some $(m_{1},m_{2},m_{3},\ldots,m_{r})\in \mathbb{Z}^{r}.$  Then by \cite[Lemma 2.1, p.463]{LT}, we have $$\sigma_{w,1}^{*} \mathcal{L}=\mathcal{O}_{Z_{w(1)}}(m_2,m_{3},\ldots, m_{r}).$$ Further, note that by \cite[Example 1.4.4, p.51]{L} $\sigma_{w,1}^{*} \mathcal{L}$ is nef. So, by induction $\sigma_{w,1}^{*} \mathcal{L}$ is globally generated on $Z_{w(1)}.$ Hence, by \cite[Corollary 3.3., p.465]{LT} it follows that $m_2,\ldots, m_r\ge 0,$  Furthermore, \cite[Lemma 2.1, p.463]{LT} also gives $$\sigma_{w,2}^{*}\mathcal{L}=\mathcal{O}_{Z_{w(2)}}(m_1,m_3,\dots,m_r)\otimes \pi_{w(2)[r-2]}^{*}(\mathcal{L}_{w(2)[r-2]}(m_2 \varpi_{i_{2}})).$$  
			Now, note that $w(2)[r-2]=(s_{i_{1}})$ and hence $Z_{w(2)[r-2]}=P_{\alpha_{i_{1}}}/B\simeq \mathbb{P}^{1},$ we identify $\mathcal{L}_{w(2)[r-2]}(\varpi_{i_{2}})$ with $\mathcal{O}_{\mathbb{P}^{1}}(\langle \varpi_{i_{2}},\alpha_{i_{1}}^{\vee}\rangle).$ When $\alpha_{i_{1}}\neq \alpha_{i_{2}},$ 
			$\pi_{w(2)[r-2]}^{*}\mathcal{L}_{w(2)[r-2]}(\varpi_{i_{2}})$ is trivial line bundle. Hence, by \cite[Example 1.4.4, p.51]{L} $\sigma_{w,2}^{*}\mathcal{L}=\mathcal{O}_{Z_{w(2)}}(m_1,m_3,\dots,m_r)$ is nef. So, by induction $\sigma_{w,2}^{*}\mathcal{L}=\mathcal{O}_{Z_{w(2)}}(m_1,m_3,\dots,m_r)$ is globally generated. Therefore, by \cite[Corollary 3.3, p.465]{LT} we get that $m_{1}\ge 0.$
			If $\alpha_{i_{1}}=\alpha_{i_{2}},$ then
			$Z_{w}\simeq P_{\alpha_{i_{1}}}/B \times Z_{w(1)}\simeq \mathbb{P}^{1} \times Z_{w(1)}.$
			Under this isomorphism $\mathcal{L}$ identifies with $\mathcal{O}_{\mathbb{P}^{1}}(m_1) \boxtimes \mathcal{O}_{Z_{w(1)}}(m_2,\ldots, m_r).$ This
			proves that $m_1\ge 0.$
		\end{proof}

		Now we recall some definitions and results of  big line bundles on a projective variety following Lazersfeld \cite{L}.
		
		Let  $X$ be an irreducible projective variety. Let $\mathcal{L}$ be a line bundle on $X.$ 
		\begin{definition}
			The {\it semigroup} of $\mathcal{L}$ consists of those non-negative powers of $\mathcal{L}$ that have a non-zero section:
			$$N(\mathcal{L}) = N(X, \mathcal{L}) =\{ m\ge 0~ | ~H^0 (X, \mathcal{L}^{\otimes m})\neq 0\}.$$ 
			
			 For a Cartier divisor $D$ one takes the {\it semigroup} $N(X, D):=N(X,\mathcal{O}_{X}(D)).$
			
		\end{definition}
		
		Given $m\in N(X,\mathcal{L}),$ consider the rational mapping $\phi_{m}: X \dashrightarrow \mathbb{P}(H^{0}(X, \mathcal{L}^{\otimes m})).$ We denote by $\phi_{m}(X)$ the closure of its image in  $\mathbb{P}(H^{0}(X, \mathcal{L}^{\otimes m}),$ i.e., the image of the closure of the graph of $\phi_{m}.$

		\begin{definition}[Iitaka Dimension]
			Assume that $X$ is a normal projective variety.	Then the	
			{\it Iitaka dimension} of $\mathcal{L}$ is defined to be
			$$\kappa(\mathcal{L}) = \kappa(X,\mathcal{L})= \begin{matrix} \text{max} & \{\dim \phi_{m}(X)\} ~\\ 
				m\in N(\mathcal{L}) & \end{matrix}$$	
			provided that $N(\mathcal{L})\neq 0.$ If $H^0 (X, \mathcal{L}^{\otimes m} )= 0$ for all $m > 0,$ one puts $\kappa(X,\mathcal{L})=-\infty.$ If $X$ is non-normal, pass to its normalization $\nu :X'\to X$ and set $\kappa(X,\mathcal{L})= \kappa(X, \nu ^{*} \mathcal{L}).$ 
			Finally, for a Cartier divisor $D$ one takes $\kappa(X, D) :=\kappa(X, \mathcal{O}_{X}(D)).$  
			
		\end{definition}
		Thus, either $\kappa(X, \mathcal{L})=-\infty,$ or else $0\le \kappa(X, \mathcal{L})\le \dim X.$

		\begin{definition}[Big]
			A line bundle $\mathcal{L}$ on the irreducible projective variety $X$ is {\em big} if $\kappa(X, \mathcal{L}) = \dim (X).$ A Cartier divisor $D$ on $X$ is {\em big} if $\mathcal{O}_{X}(D)$ is {\em big}.	
		\end{definition}
		
		\begin{lemma}\label{lem3.3}
			Let $D$ be an effective Cartier divisor on a projective variety
			$X$ and let $U:= X- {\rm Supp}(D).$ If $U$ is affine, then the line bundle $\mathcal{O}_{X}(D)$ on $X$ is big.
		\end{lemma}
		\begin{proof}
			See \cite[Lemma 2.4, p.61]{Bri}.
		\end{proof}
		
		It is well known that the anti-canonical line bundle of $Z_{w}$ is big. In \cite[Corollary 5.7, p.13]{BKS-1}, \cite[Proof of Theorem 5.3, p.2558]{Cha} this fact is stated for a reduced sequence $w$ of simple reflections but the proof works out for any sequence $w$ of simple reflections. For the sake of completeness we give a proof here.
		
		\begin{proposition}\label{prop4.15}
			The anti-canonical line bundle of $Z_{w}$ is big. 
		\end{proposition}
		\begin{proof}
			The anti-canonical line bundle $\omega_{Z_{w}}^{-1}$ of $Z_{w}$ is isomorphic to
			$\mathcal{O}_{Z_{w}}(\sum\limits_{j=1}^{r}Z_{w(j)}) \otimes \mathcal{O}_{Z_{w}}(\rho)$ (cf. \cite[Proof of Proposition 10, p.37]{MR85},
			\cite[Proposition 2.2.2, p.67]{BK05}).

			Let $s$ be a unique section (up to non-zero constants) of $\mathcal{O}_{Z_{w}}(\sum_{j=1}^{r}Z_{w(j)})$ such that ${\rm div}(s)=\sum_{j=1}^{r}Z_{w(j)}.$ Note that ${\rm div}(s)$ is an effective Cartier divisor on $Z_{w}$ such that the support $\text{Supp}({\rm div}(s))=\bigcup_{k=1}^{r} Z_{w(j)},$  Since $P_{\alpha_{i_{j}}}=B\sqcup Bs_{i_{j}}B$ for every $1\le j\le r,$ we have $$Z_{w}- \text{Supp}({\rm div}(s))=\frac{Bs_{i_{1}}B\times Bs_{i_{2}}B\times\cdots\times Bs_{i_{r}}B}{B\times B\times\cdots \times B}.$$
			
			Further, $Z_{w}- \text{Supp}({\rm div}(s))$ is isomorphic to $\prod_{j=1}^{r}U_{\alpha_{i_{j}}},$  an affine $r$-space (see \cite[Chapter 2, Section 2.2, p.65]{BK05}). Now since $Z_{w}$ is a smooth projective variety such that $Z_{w}- \text{Supp}({\rm div}(s))$ is affine, by \cref{lem3.3} ${\rm div}(s)$ is big. So, the line bundle $\mathcal{O}_{Z_{w}}({\rm div}(s))$ is big. 
			
			On the other hand, $\mathcal{O}_{Z_{w}}(\rho)$ is globally generated being pullback of the globally generated line bundle $\mathcal{L}(\rho)$ on $G/B.$ In particular, $\mathcal{O}_{Z_{w}}(\rho)$ is the line bundle corresponding to some effective divisor. Since $\omega_{Z_{w}}^{-1}=\mathcal{O}_{Z_{w}}({\rm div}(s))\otimes \mathcal{O}_{Z_{w}}(\rho),$ it follows that $\omega_{Z_{w}}^{-1}$ is big.
			
		\end{proof}
		
		\begin{corollary}\label{cor3.6} The followings are equivalent.
			\begin{itemize}
				\item [(i)] $Z_{w}$ is weak-Fano
				
				\item [(ii)]  $\omega_{Z_{w}}^{-1}$ is globally generated.
				
				\item [(iii)] $\omega_{Z_{w}}^{-1}$ is nef.
			\end{itemize}
			
		\end{corollary}
		\begin{proof}
			Clearly, by using \cref{thm4.16} it follows that $(i)\implies(iii)\iff (ii).$ On the other hand, by using \cref{prop4.15} and $(iii),$ $(i)$ follows. 
		\end{proof}

		Now we prepare in order to give necessary and sufficient conditions for $Z_{w}$ to be Fano (weak-Fano).

		\begin{lemma}
			
			Let $\lambda\in X(B).$ Let $(m_{r1},m_{r2},\ldots ,m_{rr})\in \mathbb{Z}^{r}$ be such that the line bundle
			$$\mathcal{O}_{Z_{w}}(\lambda)=\mathcal{O}_{Z_{w}}(m_{r1},m_{r2},\ldots, m_{rr})$$ (cf. \cite[Section 3.1, p.464]{LT}).
			Then we have 
			\begin{itemize}
				\item[(i)] $m_{rr}=\langle\lambda,\alpha_{i_{r}}^{\vee} \rangle.$
				
				\item[(ii)]  $m_{rk}=\langle \lambda-\sum_{l=k+1}^{r}m_{rl}\varpi_{i_{l}},\alpha_{i_{k}}^{\vee}\rangle$ for all $1\le k\le r-1.$
			\end{itemize}
		\end{lemma}
		\begin{proof}
			Let $\pi_{r}: G/B\longrightarrow G/P_{\alpha_{i_{r}}}$ denote the natural projection map. Then consider the fibre product diagram
			
			$\centerline{\xymatrixcolsep{5pc}\xymatrix{ Z_{w}\ar[r]\ar[d]_{\pi_{w[1]}}&  G/B\ar[d]_{\pi_{r}}\\
					Z_{w[1]}	\ar[r]&  G/P_{\alpha_{i_{r}}}}}.$
			
			Since $\pi_{w[1]}$ is a $\mathbb{P}^{1}$-fibration by the above diagram, we have $$\mathcal{O}_{Z_{w}}(\lambda)=\mathcal{O}_{Z_{w}}(m'_{rr})\otimes \mathcal{O}_{Z_{w[1]}}(\lambda'),$$ where $m'_{rr}=\langle \lambda,\alpha_{i_{r}}^{\vee}\rangle$ and $\lambda'=\lambda-m'_{rr}\varpi_{i_{r}}.$ 
			
			By induction we have $$\mathcal{O}_{Z_{w[1]}}(\lambda')=\mathcal{O}_{Z_{w[1]}}(m'_{r-11},m'_{r-12},\ldots,m'_{r-1r-1}),$$ where $m_{r-1k}$'s  are given in the following:
			\begin{itemize}
				\item $m'_{r-1r-1}=\langle\lambda',\alpha_{i_{r-1}}^{\vee} \rangle$
				
				\item  $m'_{r-1k}=\langle \lambda'-\sum_{l=k+1}^{r-1}m'_{r-1l}\varpi_{i_{l}},\alpha_{i_{k}}^{\vee}\rangle$ for all $1\le k\le r-2.$
			\end{itemize}
			
			Hence, we have $\mathcal{O}_{Z_{w}}(\lambda)=\mathcal{O}_{Z_{w}}(m'_{r-11},m'_{r-12},\ldots,m'_{r-1r-1}, m'_{rr}).$ Since $\{\mathcal{O}_{Z_{w[j]}}(1): 1\le j\le r\}$ forms a basis of ${\rm Pic}(Z_{w}),$ we have
			\begin{itemize}
				\item $m_{rr}=m'_{rr}=\langle \lambda,\alpha_{i_{r}}^{\vee}\rangle$
				
				\item  $m_{rk}=m'_{r-1k}=\langle \lambda-\sum_{l=k+1}^{r}m_{rl}\varpi_{i_{l}},\alpha_{i_{k}}^{\vee}\rangle$ for all $1\le k\le r-1.$ 
			\end{itemize}  
		\end{proof}

		\begin{corollary}
			Fix an integer $1\le j\le r.$ Let  $(m_{j1},m_{j2},\ldots, m_{jj})\in \mathbb{Z}^{j}$ be such that $$\mathcal{O}_{Z_{w[r-j]}}(\lambda)=\mathcal{O}_{Z_{w[r-j]}}(m_{j1},m_{j2},\ldots,m_{jj}).$$ 
			Then we have 
			\begin{itemize}
				\item[(i)] $m_{jj}=\langle\lambda,\alpha_{i_{j}}^{\vee} \rangle.$
				
				\item[(ii)]  $m_{jk}=\langle \lambda-\sum_{l=k+1}^{j}m_{jl}\varpi_{i_{l}},\alpha_{i_{k}}^{\vee}\rangle$ for all $1\le k\le j-1.$
			\end{itemize}
		\end{corollary}

		\begin{corollary}\label{cor4.12}
			Fix an integer $1\le j\le r.$ Let $(m_{j1},m_{j2},\ldots, m_{jj})\in \mathbb{Z}^{j}$ be such that $$\mathcal{O}_{Z_{w[r-j]}}(\alpha_{i_{j}})=\mathcal{O}_{Z_{w[r-j]}}(m_{j1},m_{j2},\ldots,m_{jj}).$$
			Then we have 
			\begin{itemize}
				\item[(i)] $m_{jj}=\langle \alpha_{i_{j}},\alpha_{i_{j}}^{\vee}\rangle =2$
				
				\item [(ii)] $m_{jk}=\langle \alpha_{i_{j}}-\sum_{l=k+1}^{j}m_{jl}\varpi_{i_{l}},\alpha_{i_{k}}^{\vee}\rangle$ for all $1\le k\le j-1.$
			\end{itemize}
		\end{corollary}

		\begin{lemma}\label{lem4.13}
			 $\omega_{Z_{w}}^{-1}$ is isomorphic to
			$\mathcal{O}_{Z_{w[r-1]}}(\alpha_{i_{1}})\otimes \cdots \otimes\mathcal{O}_{Z_{w[1]}}(\alpha_{i_{r-1}})\otimes \mathcal{O}_{Z_{w}}(\alpha_{i_{r}}).$
		\end{lemma}
		\begin{proof}
			Consider the fibre product diagram
			
			$\centerline{\xymatrixcolsep{5pc}\xymatrix{ Z_{w}\ar[r]\ar[d]_{\pi_{w[1]}}& G/B\ar[d]_{\pi_{r}}\\
					Z_{w[r-1]}	\ar[r]&  G/P_{\alpha_{i_{r}}}}}.$
			
			Then we have the following short exact sequence 
			\begin{equation}\label{eq4.3}
				0\longrightarrow \mathcal{R}_{\pi_{w[1]}}\longrightarrow T_{Z_{w}}\longrightarrow \pi_{w[1]}^{*} T_{Z_{w[1]}}\longrightarrow 0,
			\end{equation}
		of tangent bundles on $Z_{w},$ where $T_{Z_{w}}$ (respectively, $T_{Z_{w[1]}}$) denotes the tangent bundle of $Z_{w}$ (respectively, of $Z_{w[1]}$), and $\mathcal{R}_{\pi_{w[1]}}$ denotes the relative tangent bundle with respect to the map $\pi_{w[1]}.$ Hence, by \eqref{eq4.3} we have $\omega_{Z_{w}}^{-1}=\mathcal{R}_{\pi_{w[1]}}\otimes \pi_{w[1]}^{*}\omega_{Z_{w[1]}}^{-1}.$
			
			Since $\mathcal{L}(\alpha_{i_{r}})$ is the relative tangent bundle on $G/B$  with respect to $\pi_{r},$ the pull back line bundle $\mathcal{O}_{Z_{w}}(\alpha_{i_{r}})$ is equal to  $\mathcal{R}_{\pi_{w[1]}}$ on $Z_{w}.$
			
			By induction we have $\omega_{Z_{w[1]}}^{-1}=\mathcal{O}_{Z_{w[1]}}(\alpha_{i_{r-1}})\otimes \mathcal{O}_{Z_{w[1][1]}}(\alpha_{i_{r-2}})\otimes\cdots\otimes \mathcal{O}_{Z_{w[1][r-2]}}(\alpha_{i_{1}}).$ Therefore, we have $\omega_{Z_{w}}^{-1}=\mathcal{O}_{Z_{w}}(\alpha_{i_{r}})\otimes \mathcal{O}_{Z_{w[1]}}(\alpha_{i_{r-1}})\otimes \mathcal{O}_{Z_{w[2]}}(\alpha_{i_{r-2}})\otimes\cdots\otimes \mathcal{O}_{Z_{w[r-1]}}(\alpha_{i_{1}}).$
		\end{proof}
		Recall that by \cite[Section 3.1, p.464]{LT} we have $$\omega_{Z_{w}}^{-1}=\mathcal{O}_{Z_{w}}(n_{1},n_{2},\ldots,n_{r})$$ for a unique $(n_{1},\ldots, n_{r})\in \mathbb{Z}^{r}.$ 
		
		\begin{proposition}\label{prop4.14}
			
			Then we have 
			\begin{itemize}
				\item[(i)] $n_{r}=2.$
				
				\item [(ii)] $n_{j}=\sum_{l=j}^{r}m_{lj}$ for all $1\le j\le r-1.$
			\end{itemize} 
		where $m_{lj}$ are as in the statement of $\cref{cor4.12}.$
		\end{proposition}
		\begin{proof}
			By \cref{lem4.13} we have  $$\omega_{Z_{w}}^{-1}=\mathcal{O}_{Z_{w[r-1]}}(\alpha_{i_{1}})\otimes \cdots \otimes\mathcal{O}_{Z_{w[1]}}(\alpha_{i_{r-1}})\otimes \mathcal{O}_{Z_{w}}(\alpha_{i_{r}}).$$ 
			
			For $1\le j\le r,$ by using \cref{cor4.12} we have $$\mathcal{O}_{Z_{w[r-j]}}(\alpha_{i_{j}})=\mathcal{O}_{Z_{w[r-j]}}(m_{j1},m_{j2},\ldots,m_{jj}),$$ where $m_{jj}=2,$ $m_{jk}=\langle \alpha_{i_{j}}-\sum_{l=k+1}^{j}m_{jl}\varpi_{i_{l}},\alpha_{i_{k}}^{\vee}\rangle$ for all $1\le k\le j-1.$ Therefore, we have $$\omega_{Z_{w}}^{-1}=\mathcal{O}_{Z_{w}}(\sum_{l=1}^{r}m_{l1},\ldots,\sum_{l=j}^{r}m_{lj},\ldots,m_{rr}).$$ Since $\{\mathcal{O}_{Z_{w[j]}}(1): 1\le j\le r\}$ forms a basis of \text{Pic}($Z_{w}$), we have  $n_{r}=m_{rr}=2$ and $n_{j}=\sum_{l=j}^{r}m_{lj}$ for all $1\le j\le r-1$.
		\end{proof}

		\begin{corollary}\label{cor4.17}
			 $\omega_{Z_{w}}^{-1}$ is globally generated $($respectively, ample$)$ if and only if  $\sum_{l=j}^{r}m_{lj}\ge 0$ $($\text{respectively,} $\sum_{l=j}^{r}m_{lj} >0$$)$ for all $1\le j\le r-1.$
		\end{corollary}
		\begin{proof}
			By using \cref{prop4.14}, and \cite[Theorem 3.1, Corollary 3.3, p.464-465]{LT} proof of the corollary follows.
		\end{proof}

		\begin{proof}[Proof of \cref{cor4.18}]
			By \cref{cor4.17} it follows that $Z_{w}$ is Fano if and only if  $\sum_{l=j}^{r}m_{lj} >0$ for all $1\le j\le r-1.$ On the other hand, by \cref{cor4.17} and \cref{cor3.6} it follows that $Z_{w}$ is weak-Fano if and only if  $\sum_{l=j}^{r}m_{lj}\ge 0$ for all $1\le j\le r-1.$
		\end{proof}

		\section{Line bundles on $G$-BSDH-varieties}\label{sec4}
		In this section we first define $\mathcal{O}(1)$-basis of the Picard group ${\rm Pic}(\widetilde{Z_{w}})$ of $\widetilde{Z_{w}}.$ Using $\mathcal{O}(1)$-basis of ${\rm Pic}(\widetilde{Z_{w}})$ we give a characterization of ample, very ample and globally generated line bundles on $\widetilde{Z_{w}}.$ We prove that every nef line bundle on $\widetilde{Z_{w}}$ is globally generated. 
		\subsection{Picard group of $G$-BSDH-variety}
		Let $\widetilde{\pi_{r}}: G/B\times G/B \longrightarrow G/B\times G/P_{\alpha_{i_{r}}}$ denote the natural projection map. Note that we have the following commutative fibre product diagram:
		
		$\centerline{\xymatrixcolsep{5pc}\xymatrix{\widetilde{Z_{w}}\ar[r]^{\widetilde{\varphi_{w}}}\ar[d]_{\widetilde{\pi_{w[1]}}}& G/B\times G/B\ar[d]^{\widetilde{\pi_{r}}}\\
				\widetilde{Z_{w[1]}}\ar[r]_{\widetilde{\varphi_{w[1]}}}& G/B\times G/P_{\alpha_{i_{r}}}}}$

		For $B$-modules $M$ and $N,$ we denote the vector bundle $p_{1}^{*}\mathcal{L}(M)\otimes p_{2}^{*}\mathcal{L}(N)$ on $G/B\times G/B$ by $\mathcal{L}(M\boxtimes N),$ where $p_{i}$ denotes the projection from $G/B\times G/B$ onto the $i^{\rm th}$ factor. We denote the pull back vector bundle $\widetilde{\varphi_{w}}^{*}\mathcal{L}(M\boxtimes N)$ on $\widetilde{Z_{w}}$ by simply $\mathcal{O}_{\widetilde{Z_{w}}}(M\boxtimes N).$
		In particular, for $\lambda, \mu\in X(B),$ $\mathcal{L}(\lambda\boxtimes \mu)$ denotes the line bundle on $G/B\times G/B$ associated to $\lambda, \mu,$ and its pull back on $\widetilde{Z_{w}}$ is denoted by $\mathcal{O}_{\widetilde{Z_{w}}}(\lambda\boxtimes\mu).$   
		
		Let ${\varepsilon}_{0}$ denote the trivial character of $B.$ Then for any $B$-module $V,$ we define $\mathcal{O}_{\widetilde{Z_{w}}}(V):= \widetilde{\varphi_{w}}^{*}\mathcal{L}(\varepsilon_{0}\boxtimes V).$ In particular, for $\lambda \in X(B),$ we have  $\mathcal{O}_{\widetilde{Z_{w}}}(\lambda):= \widetilde{\varphi_{w}}^{*}\mathcal{L}(\varepsilon_{0}\boxtimes \lambda).$ 
		
		Recall that $V_{\alpha_{i_{r}}}=H^{0}(P_{\alpha_{i_{r}}} /B, \mathcal{L}(\varpi_{i_{r}}))$ denotes the natural two dimensional $B$-module.  Then we have the following exact sequences
		\begin{equation*}
			0\to L_{\varepsilon_{0}}\boxtimes L_{s_{i_{r}}(\varpi_{i_{r}})}\to L_{\varepsilon_{0}} \boxtimes V_{\alpha_{i_{r}}}\to L_{\varepsilon_{0}}\boxtimes L_{\varpi_{i_{r}}}\to 0
		\end{equation*}	
	 of $B\times B$-modules, where $L_{\lambda}\boxtimes L_{\mu}$ denotes the one-dimensional representation of $B\times B$ with character $(\lambda,\mu).$
		This gives an exact sequence 
		\begin{equation*}
			0\to  \mathcal{L}({\varepsilon_{0}}\boxtimes s_{i_{r}}(\varpi_{i_{r}}))\to \mathcal{L}(\varepsilon_{0}\boxtimes V_{\alpha_{i_{r}}})\to  \mathcal{L}(\varepsilon_{0}\boxtimes \varpi_{i_{r}}) \to 0
		\end{equation*}
	of vector bundles on $G/B\times G/B.$ Pulling back by $\widetilde{\varphi_{w[1]}}$ induces an exact sequence 
		\begin{equation*}
			0\to \mathcal{O}_{\widetilde{Z_{w[1]}}}(s_{i_{r}(\varpi_{i_{r}})})\to \mathcal{O}_{\widetilde{Z_{w[1]}}}(V_{\alpha_{i_{r}}})\to  \mathcal{O}_{\widetilde{Z_{w[1]}}}(\varpi_{i_{r}}) \to 0 
		\end{equation*}
	of vector bundles on $\widetilde{Z_{w[1]}}.$ Then $\widetilde{Z_{w}}$ is identified with the projective bundle $\mathbb{P}(\mathcal{O}_{\widetilde{Z_{w[1]}}}(V_{\alpha_{i_{r}}})).$ Thus the corresponding universal quotient bundle of $\mathcal{O}_{\widetilde{Z_{w[1]}}}(V_{\alpha_{i_{r}}})$ is $\mathcal{O}_{\widetilde{Z_{w[1]}}}(\varpi_{i_{r}}).$  So, we define $\mathcal{O}_{\widetilde{Z_{w}}}(1):=\mathcal{O}_{\widetilde{Z_{w}}}(\varpi_{i_{r}})$ and write $\mathcal{O}_{\widetilde{Z_{w}}}(m):= \mathcal{O}_{\widetilde{Z_{w}}}(1)^{\otimes m}.$ Note that $\mathcal{O}_{\widetilde{Z_{w}}}(1)$ is a line bundle on $\widetilde{Z_{w}}$ with degree one along the fibers of $\widetilde{\pi_{w[1]}}.$

		Since the map $\widetilde{\pi_{w[1]}}:\widetilde{Z_{w}}\to \widetilde{Z_{w[1]}}$ is a $\mathbb{P}^{1}$-fibration, we have Pic$(\widetilde{Z_{w}})=$ Pic($\widetilde{Z_{w[1]}}$) $\oplus \mathbb{Z}\mathcal{O}_{\widetilde{Z_{w}}}(1).$ Now by induction on the number of $\mathbb{P}^{1}$-fibrations, we have $${\rm Pic}(\widetilde{Z_{w}})=\mathbb{Z}\mathcal{O}_{\widetilde{Z_{w}}}(1)\oplus \mathbb{Z}\mathcal{O}_{\widetilde{Z_{w[1]}}}(1)\oplus \cdots \oplus \mathbb{Z}\mathcal{O}_{\widetilde{Z_{w[r-1]}}}(1)\oplus \text{Pic}(G/B),$$ where we write $\mathcal{O}_{\widetilde{Z_{w[j]}}}(1)$ instead of the pull back $\widetilde{\pi_{w[j]}}^{*}\mathcal{O}_{\widetilde{Z_{w[j]}}}(1).$ Thus,  ${\rm Pic}(\widetilde{Z_{w}})$ is a free abelian group of rank $n+r.$ This gives inductively a basis for ${\rm Pic}(\widetilde{Z_{w}})$ which we call the $\mathcal{O}(1)$-basis. For $(m_{1},m_{2},\ldots,m_{r})\in \mathbb{Z}^{r},$ we denote the line bundle $\mathcal{O}_{\widetilde{Z_{w}}}(m_{r})\otimes \mathcal{O}_{\widetilde{Z_{w[1]}}}(m_{r-1})\otimes \cdots \otimes \mathcal{O}_{\widetilde{Z_{w[r-1]}}}(m_{1})$ by $\mathcal{O}_{\widetilde{Z_{w}}}(m_{1},\ldots, m_{r}).$

		\cref{lem4.1} follows immediately from the above discussion.
		 
		\begin{proof}[Proof of \cref{thm3.1}]
			We prove by induction on $r=\dim(Z_{w}).$
			If $r=0,$ then $Z_{w}$ is a point and hence $\widetilde{Z_{w}}\longrightarrow G/B$ is an isomorphism. So, we are done. 
			
			Now, assume that $r\ge 1.$
			Then the $\mathbb{P}^1$-bundle $\widetilde{\pi_{w[1]}}: \widetilde{Z_{w}}\longrightarrow \widetilde{Z_{w[1]}}$ may by identified with the projective bundle $\mathbb{P}(\mathcal{V})\to \widetilde{Z_{w[1]}},$ where $\mathcal{V}$ is the rank two vector bundle $\mathcal{O}_{\widetilde{Z_{w[1]}}}(\varepsilon_{0}\boxtimes V_{\alpha_{i_{r}}})$ on $\widetilde{Z_{w[1]}}$ and $\mathcal{O}_{\mathbb{P}(\mathcal{V})}=\mathcal{O}_{\widetilde{Z_{w}}}(1).$ Note that $\mathcal{L}(\varepsilon_{0}\boxtimes V_{\alpha_{i_{r}}} )$ is a globally generated vector bundle on $G/B\times G/B.$ Since $\mathcal{V}$ is the pull back of a globally
			generated vector bundle on $G/B\times G/B,$ it is also globally generated. Thus, we have the following commutative diagram:
			
			$\centerline{\xymatrixcolsep{5pc}\xymatrix{\widetilde{Z_{w}}\ar@{^{(}->}[r]^{\varphi}\ar[d]& \mathbb{P}^{N}\times \widetilde{Z_{w[1]}}\ar[ld]\\
					\widetilde{Z_{w[1]}}}}$
			for some $N\in \mathbb{N}$ such that $\varphi$ is closed embedding and $\mathcal{O}_{\widetilde{Z_{w}}}(1)=\varphi^{*}(\mathcal{O}_{\mathbb{P}^{N}}(1)\boxtimes \mathcal{O}_{\widetilde{Z_{w[1]}}}).$ Therefore, it follows that $\varphi^{*}(\mathcal{O}_{P^{N}}(n)\boxtimes\widetilde{\pi_{w[1]}}^{*}\mathcal{L}')$ is very ample whenever $n>0$ and $\mathcal{L}'$ is very ample on $\widetilde{Z_{w[1]}}.$ Hence, by induction $\mathcal{L}$ is very ample line bundle on $\widetilde{Z_{w}}$ if $m_{1},\ldots, m_{r}>0$ and $\lambda$ is regular dominant.
			
			Conversely, assume that $\mathcal{L}$ is a very ample line bundle on $\widetilde{Z_{w}}.$ 	Fix an $\alpha\in S.$ Consider the following commutative diagram:
			
			$$\centerline{\xymatrixcolsep{5pc}\xymatrix{P_{\alpha}\times^{B}Z_{w}\ar@{^{(}->}[r]^{\iota}\ar[d]_{\phi}& \widetilde{Z_{w}}\ar[d]^{\widetilde{ \pi_{w[r]}}}\\
					P_{\alpha}/B	\ar@^{{(}->}[r]_{\iota}& G/B}}$$
			Then the restriction of $\mathcal{L}$ to $P_{\alpha}\times^{B}Z_{w}$ is very ample. Note that $\mathcal{L}|_{P_{\alpha}\times^{B}Z_{w}}=\mathcal{O}_{Z_{w}}(m_{1},\ldots,m_{r})\otimes\iota^{*}\widetilde{\pi_{w[r]}}^{*}\mathcal{L}(\lambda)$ and  $\iota^{*} \widetilde{\pi_{w[r]}}^{*}\mathcal{L}(\lambda)=\,\phi^{*}\mathcal{O}_{P_{\alpha}/B}(\langle \lambda,\alpha^{\vee} \rangle).$ Observe that $P_{\alpha}\times^{B} Z_{w}= Z_{w'},$ where $w'$ is the sequence $(s_{\alpha}, w).$ Hence, by \cite[Theorem 3.1, p.464]{LT} it follows that $m_{1},\ldots, m_{r}> 0$ and $\langle \lambda,\alpha^{\vee}\rangle>0.$ Therefore, $m_{1},\ldots, m_{r}> 0$ and $\lambda$ is regular dominant.
		\end{proof}
		
		\begin{corollary}\label{cor4.3}
			Any ample line bundle on $\widetilde{Z_{w}}$ is very ample.
		\end{corollary}
		
		\begin{corollary}\label{cor3.3}
		Let $(m_{1},m_{2},\ldots,m_{r})\in \mathbb{Z}^{r}$ and $\lambda\in X(B).$ Then	a line bundle $\mathcal{L}= \mathcal{O}_{\widetilde{Z_{w}}}(m_{1},\ldots, m_{r})\otimes\widetilde{\pi_{w[r]}}^{*}\mathcal{L}(\lambda)$ on $\widetilde{Z_{w}}$ is globally generated if and only if $m_{1},\dots,m_{r}\ge 0$ and $\lambda$ is dominant. 
		\end{corollary}
		\begin{proof}
			Assume that $m_{1},\dots,m_{r}\ge 0$ and $\lambda$ is dominant. Then clearly $\mathcal{L}$ is a globally generated line bundle on $\widetilde{Z_{w}}.$
			
			Conversely, assume that $\mathcal{L}$ is a globally generated line bundle on $\widetilde{Z_{w}}.$  Fix an $\alpha\in S.$ Consider the following commutative diagram:
			
			$$\centerline{\xymatrixcolsep{5pc}\xymatrix{P_{\alpha}\times^{B}Z_{w}\ar@{^{(}->}[r]^{\iota}\ar[d]_{\phi}& \widetilde{Z_{w}}\ar[d]^{\widetilde{ \pi_{w[r]}}}\\
					P_{\alpha}/B	\ar@^{{(}->}[r]^{\iota}& G/B}}$$
			
			Then the restriction of $\mathcal{L}$ to $P_{\alpha}\times^{B}Z_{w}$ is also globally generated. Note that $\mathcal{L}|_{P_{\alpha}\times^{B}Z_{w}}=\mathcal{O}_{Z_{w}}(m_{1},\ldots,m_{r})\otimes\iota^{*}\widetilde{\pi_{w[r]}}^{*}\mathcal{L}(\lambda)$ and $\iota^{*}\widetilde{ \pi_{w[r]}}^{*}\mathcal{L}(\lambda)=\phi^{*}\mathcal{O}_{P_{\alpha}/B}(\langle \lambda,\alpha^{\vee} \rangle).$ Hence, by \cite[Corollary 3.3, p.465]{LT} it follows that $m_{1},\ldots, m_{r} \ge0$ and $\langle \lambda,\alpha^{\vee}\rangle\ge 0.$ Therefore, we have $m_{1},\ldots, m_{r}\ge 0$ and $\lambda$ is dominant. 
		\end{proof}

		\begin{proof}[Proof of \cref{thm3.5}]
			Assume that $\mathcal{L}$ is globally generated. Then by \cite[Example 1.4.5., p.52]{L} $\mathcal{L}$ is nef. 
			
			Conversely, suppose that $\mathcal{L}$ is nef. 
			By \cref{lem4.1} we have $$\mathcal{L}=\mathcal{O}_{\widetilde{Z_{w}}}(m_{1},\ldots,m_{r})\otimes \widetilde{\pi_{w[r]}}^{*}\mathcal{L}(\lambda)$$ for a unique $(m_{1},m_{2},\ldots,m_{r})\in \mathbb{Z}^{r}$ and a unique $\lambda\in X(B).$

			Fix an $\alpha\in S.$ Consider the following commutative diagram:
			
			$$\centerline{\xymatrixcolsep{5pc}\xymatrix{P_{\alpha}\times^{B}Z_{w}\ar@{^{(}->}[r]^{\iota}\ar[d]_{\phi}& \widetilde{Z_{w}}\ar[d]^{\widetilde{ \pi_{w[r]}}}\\
					P_{\alpha}/B	\ar@^{{(}->}[r]& G/B}}$$
			Then the restriction of $\mathcal{L}$ to $P_{\alpha}\times^{B}Z_{w}$ is also nef. Note that $\mathcal{L}|_{P_{\alpha}\times^{B}Z_{w}}=\mathcal{O}_{Z_{w}}(m_{1},\ldots,m_{r})\otimes \iota^{*} \widetilde{\pi_{w[r]}}^{*}\mathcal{L}(\lambda)|_{P_{\alpha}\times^{B} Z_{w}}$ is nef. Hence, by \cref{thm4.16} it follows that $\mathcal{L}|_{P_{\alpha}\times^{B}Z_{w}}$ is globally generated. Further, note that $ \iota^{*}\widetilde{\pi_{w[r]}}^{*}\mathcal{L}(\lambda)|_{P_{\alpha}\times^{B} Z_{w}}=\phi^{*} \mathcal{O}_{P_{\alpha}/B}(\langle \lambda,\alpha ^{\vee}\rangle).$ Hence,  by \cite[Corollary 3.3, p.465]{LT} it follows that  $m_1,\ldots ,m_{r}\ge 0,$ and $\langle \lambda,\alpha^{\vee}\rangle\ge 0.$ Therefore, by \cref{cor3.3} it follows that $\mathcal{L}$ is a globally generated line bundle on $\widetilde{Z_{w}}.$  
		\end{proof}
		
		\begin{corollary}
			If $\mathcal{L}$ is an ample line bundle on $\widetilde{Z_{w}},$ then all the higher cohomologies of $\mathcal{L}$ vanish, i.e., $H^{j}(\widetilde{Z_{w}},\mathcal{L})=0$ for $j\ge 1.$
		\end{corollary}
		\begin{proof}
			By \cite[Proposition 20, p.20]{Mag} it follows that $\widetilde{Z_{w}}$ is Frobenious split. Therefore, by \cite[Theorem 17, p.347]{Mag} proof of the corollary follows. 
		\end{proof}
		
		\begin{corollary}
			If $\mathcal{L}$ is a globally generated line bundle on $\widetilde{Z_{w}},$ then  all the higher cohomologies of $\mathcal{L}$ vanish, i.e., $H^j(\widetilde{Z_{w}},\mathcal{L})=0$ for $j\ge 1.$ 
		\end{corollary}
		\begin{proof}
			By using \cite[Lemma 23, p.350]{Mag} proof of the corollary follows.
		\end{proof}

		\subsection{Big anti-canonical line bundle of $\widetilde{Z_{w}}$}
		Let $F$ be a smooth projective $B$-variety. Then consider the variety $X= G\times^{B} F$ and let $\pi: X\longrightarrow G/B$ be the natural projection map defined by $\pi([g,f])=gB$ for $g\in G$ and $f\in F.$ Let $D'$ be a $B$-invariant effective divisor on $F.$  Let $\widetilde{D'}:= G\times^{B} D'$ denote the $G$-invariant effective divisor on $X.$ 
		
		Let $p: G\longrightarrow G/B$ be the natural projection map.
		Let $\widetilde{D''}:=\pi^{-1}(\partial G/B)=p^{-1}(\partial G/B)\times^{B} F,$ where $\partial G/B$ denotes the boundary divisor of $G/B.$ Let $\widetilde{D}:= \widetilde{D'}+\widetilde{D''}.$ Note that $\widetilde{D}$ is an effective divisor on $X.$

		\begin{proof}[Proof of \cref{lem4.8}]
				
			 Note that the support $\text{Supp}(\widetilde{D})=\text{Supp}(\widetilde{D'})\bigcup \text{Supp}(\widetilde{D''}).$ Therefore, we have $X- \text{Supp}(\widetilde{D})=(G- p^{-1}(\partial G/B))\times^{B} (F- D' ).$ 
			 Consider the restriction of $\pi$ to $(G- p^{-1}(\partial G/B))\times^{B} (F- D').$ 
			 Now since both the varieties $G/B- \partial G/B$ and $F- D'$  are affine, the variety $X- \text{Supp}(\widetilde{D})$ is also affine. Therefore, by using \cref{lem3.3} it follows that $\widetilde{D}$ is big.
		\end{proof}
		
		\begin{rem}
			By \cref{lem3.3} it follows that the divisor $D'$ is big on $F.$
		\end{rem}
		\begin{rem}
			The divisor $\widetilde{D'}$ itself need not be big. For example, take $F=G/B,$ $D'=\partial G/B$ and the action of $B$ on $G/B$ is trivial. Then  $X=G/B\times G/B$ and the divisor   $\widetilde{D'}$ on $X$ becomes $G/B\times \partial G/B.$ The divisor $\widetilde{D'}$ on $G/B\times G/B$ is not big because any positive multiple of $\widetilde{D'}$ can not be the sum of an ample divisor and an effective divisor on $G/B\times G/B$ (cf. \cite[Corollary 2.2.7, p.141]{L}).	
		\end{rem}

		\begin{theorem}\label{thm4.8}
			The anti-canonical line bundle $\omega _{\widetilde{Z_{w}}}^{-1}$ of $\widetilde{Z_{w}}$ is big.
		\end{theorem}
		\begin{proof}
			By \cite[Proposition 2, p.356]{MR88} we have $$\omega_{\widetilde{Z_{w}}}^{-1}=\mathcal{O}_{\widetilde{Z_{w}}}( \widetilde{\partial Z_{w}})\otimes \mathcal{O}_{\widetilde{Z_{w}}}(\rho\boxtimes \rho),$$  where $\widetilde{\partial Z_{w}}:=G\times^{B} \partial Z_{w}.$  Note that  $\mathcal{O}_{\widetilde{Z_{w}}}(\rho\boxtimes\rho)=\mathcal{O}_{\widetilde{Z_{w}}}(\rho\boxtimes \varepsilon_{0})\otimes \mathcal{O}_{\widetilde{Z_{w}}}(\varepsilon_{0}\boxtimes \rho).$
			Since $\pi_{*}\mathcal{O}_{\widetilde{Z_{w}}}=\mathcal{O}_{G/B},$ there is a natural $G$-equivariant isomorphism of $H^0(\widetilde{Z_{w}}, \mathcal{O}_{\widetilde{Z_{w}}}(\rho\boxtimes \varepsilon_{0}))$ and $H^0(G/B, \mathcal{L}(\rho)).$ Let $s$ be a unique (up to  non-zero constants) $B$-eigenvector  of $H^0(G/B,\mathcal{L}(\rho))$ such that ${\rm div}(s)=\partial G/B.$ By the proof of \cref{prop4.15}, the variety $Z_{w}- \partial Z_{w}$ is affine. Thus, by \cref{lem4.8} it follows that the effective divisor $\widetilde{\partial Z_{w}}+ p^{-1}(\partial G/B)\times^{B} Z_{w}$ is big on $\widetilde{Z_{w}}.$ Hence, the line bundle $\mathcal{O}_{\widetilde{Z_{w}}}( \widetilde{\partial Z_{w}})\otimes \mathcal{O}_{\widetilde{Z_{w}}}(\rho\boxtimes \varepsilon_{0})$ is big on $\widetilde{Z_{w}}.$ Since $\mathcal{O}_{\widetilde{Z_{w}}}( \varepsilon_{0}\boxtimes \rho)$ is pull back of the globally generated line bundle $\mathcal{L}(\varepsilon_{0}\boxtimes \rho)$ on $G/B \times G/B,$ it is also globally generated. Finally, since  $\omega_{\widetilde{Z_{w}}}^{-1}=\mathcal{O}_{\widetilde{Z_{w}}}( \widetilde{\partial Z_{w}})\otimes \mathcal{O}_{\widetilde{Z_{w}}}(\rho\boxtimes \varepsilon_{0})\otimes \mathcal{O}_{\widetilde{Z_{w}}}(\varepsilon_{0}\boxtimes \rho),$ it follows that  $\omega_{\widetilde{Z_{w}}}^{-1}$ is big.
		\end{proof}

		\begin{corollary}\label{cor4.10} The followings are equivalent.
			\begin{itemize}
				\item [(i)] $\widetilde{Z_{w}}$ is weak-Fano.
				
				\item [(ii)]  $\omega_{\widetilde{Z_{w}}}^{-1}$  is globally generated.
				
				\item [(iii)]  $\omega_{\widetilde{Z_{w}}}^{-1}$ is nef.
			\end{itemize}
			
		\end{corollary}
		\begin{proof}
			Clearly, by using \cref{thm3.5} it follows that $(i)\implies(iii)\iff (ii).$ On the other hand, by using \cref{thm4.8} and $(iii),$ $(i)$ follows. 
		\end{proof}

		\section{(weak)-Fano $G$-BSDH-varieties}\label{sec5}
		In this section for an arbitrary sequence $w=(s_{i_{1}},s_{i_{2}},\ldots, s_{i_{r}})$ of simple reflections in $W$ we give necessary and sufficient conditions for $\widetilde{Z_{w}}$ to be (weak)-Fano variety.

		Let $\lambda\in X(B).$ By \cref{lem4.1} we have		$$\mathcal{O}_{\widetilde{Z_{w}}}(\lambda)=\mathcal{O}_{\widetilde{Z_{w}}}(m_{r1},m_{r2},\ldots, m_{rr})\otimes\widetilde{\pi_{w[r]}}^{*}\mathcal{L}(\lambda_{r})$$ for a unique tuple $(m_{r1},m_{r2},\ldots ,m_{rr})\in \mathbb{Z}^{r}$ and a unique $\lambda_{r}\in X(B),$ where $m_{rj}$ and $\lambda_{r}$ are given in following.
		
		\begin{lemma}
		
			\begin{itemize}
				\item[(i)] $m_{rr}=\langle\lambda,\alpha_{i_{r}}^{\vee} \rangle.$
				
				\item[(ii)]  $m_{rk}=\langle \lambda-\sum_{l=k+1}^{r}m_{rl}\varpi_{i_{l}},\alpha_{i_{k}}^{\vee}\rangle$ for all $1\le k\le r-1.$
				
				\item[(iii)] $\lambda_{r}=\lambda-\sum_{l=1}^{r}m_{rl}\varpi_{i_{l}}.$
			\end{itemize}
		\end{lemma}
		\begin{proof}
			Consider the fibre product diagram
			
			$$\centerline{\xymatrixcolsep{5pc}\xymatrix{ \widetilde{Z_{w}}\ar[r]\ar[d]_{\widetilde{\pi_{w[1]}}}& G/B\times G/B\ar[d]^{\widetilde{\pi_{r}}}\\
					\widetilde{Z_{w[1]}}	\ar[r]& G/B\times G/P_{\alpha_{i_{r}}}}}.$$
			
			Since $\widetilde{\pi_{w[1]}}$ is a $\mathbb{P}^{1}$-fibration, by the above commutative diagram we have $$\mathcal{O}_{\widetilde{Z_{w}}}(\lambda)=\mathcal{O}_{\widetilde{Z_{w}}}(m'_{rr})\otimes \mathcal{O}_{\widetilde{Z_{w[1]}}}(\lambda'),$$ where $m'_{rr}=\langle \lambda,\alpha_{i_{r}}^{\vee}\rangle $ and  
			$\lambda'=\lambda-m_{rr}\varpi_{i_{r}}.$
			
			By induction we have $$\mathcal{O}_{\widetilde{Z_{w[1]}}}(\lambda')=\mathcal{O}_{\widetilde{Z_{w[1]}}}(m'_{r-11},m'_{r-12},\ldots,m'_{r-1r-1})\otimes\widetilde{\pi_{w[r-1]}}^{*}\mathcal{L}(\lambda'_{r-1}),$$ where $m_{r-1k}$'s and $\lambda'_{r-1}$ are given in the following:
			\begin{itemize}
				\item $m'_{r-1r-1}=\langle\lambda',\alpha_{i_{r-1}}^{\vee} \rangle$
				
				\item  $m'_{r-1k}=\langle \lambda'-\sum_{l=k+1}^{r-1}m'_{r-1l}\varpi_{i_{l}},\alpha_{i_{k}}^{\vee}\rangle$ for all $1\le k\le r-2$
				
				\item $\lambda'_{r-1}=\lambda'-\sum_{l=1}^{r-1}r'_{r-1l}\varpi_{i_{r}}.$
			\end{itemize}
			
			Hence, we have $\mathcal{O}_{\widetilde{Z_{w}}}(\lambda)=\mathcal{O}_{\widetilde{Z_{w}}}(m'_{r-11},m'_{r-12},\ldots,m'_{r-1r-1}, m'_{rr})\otimes\widetilde{\pi_{w[r]}}^{*}\mathcal{L}(\lambda'_{r-1}).$ Since $\{\mathcal{O}_{\widetilde{Z_{w[j]}}}(1): 1\le j\le r\}\bigcup\{\widetilde{\pi_{w[r]}}^{*}\mathcal{L}(\varpi_{i}): 1\le i\le n\}$ forms a basis of \text{Pic}($\widetilde{Z_{w}}$), we have 
			\begin{itemize}
				\item $m_{rr}=m'_{rr}=\langle \lambda,\alpha_{i_{r}}^{\vee}\rangle$
				
				\item $m_{rk}=m'_{r-1k}=\langle \lambda-\sum_{l=k+1}^{r}m_{rl}\varpi_{i_{l}},\alpha_{i_{k}}^{\vee}\rangle$ for all $1\le k\le r-1$
				
				\item $\lambda_{r}=\lambda'_{r-1}=\lambda-\sum_{l=1}^{r}m_{rl}\varpi_{i_{j}}.$
			\end{itemize}
			
		\end{proof}
		
		Fix an integers $1\le j\le r.$ By \cref{lem4.1} we have $$\mathcal{O}_{\widetilde{Z_{w[r-j]}}}(\lambda)=\mathcal{O}_{\widetilde{Z_{w[r-j]}}}(m_{j1},m_{j2},\ldots,m_{jj})\otimes\widetilde{\pi_{w[r-j][j]}}^{*}\mathcal{L}(\lambda_{j})$$ for a unique $(m_{j1},m_{j2},\ldots, m_{jj})\in \mathbb{Z}^{j}$ and a unique $\lambda_{j}\in X(B).$ 	
		\begin{corollary} Then we have
	
			\begin{itemize}
				\item[(i)] $m_{jj}=\langle\lambda,\alpha_{i_{j}}^{\vee}\rangle.$
				
				\item[(ii)]  $m_{jk}=\langle \lambda-\sum_{l=k+1}^{j}m_{jl}\varpi_{i_{l}},\alpha_{i_{k}}^{\vee}\rangle$ for all $1\le k\le j-1.$
				
				\item[(iii)] $\lambda_{j}=\lambda-\sum_{l=1}^{j}m_{jl}\varpi_{i_{l}}.$
			\end{itemize}
		\end{corollary}

		Fix an integer $1\le j\le r.$ By \cref{lem4.1} we have $$\mathcal{O}_{\widetilde{Z_{w[r-j]}}}(\alpha_{i_{j}})=\mathcal{O}_{\widetilde{Z_{w[r-j]}}}(m_{j1},m_{j2},\ldots,m_{jj})\otimes\widetilde{\pi_{w[r-j][j]}}^{*}\mathcal{L}(\lambda_{j})$$ for a unique $(m_{j1},m_{j2},\ldots, m_{jj})\in \mathbb{Z}^{j}$ and a unique $\lambda_{j}\in X(B).$
		\begin{corollary}\label{cor4.2}
		Then we have
			\begin{itemize}
				\item[(i)] $m_{jj}=2.$
				
				\item [(ii)] $m_{jk}=\langle \alpha_{i_{j}}-\sum_{l=k+1}^{j}m_{jl}\varpi_{i_{l}},\alpha_{i_{k}}^{\vee}\rangle$ for all $1\le k\le j-1.$
				
				\item[(iii)] $\lambda_{j}=\alpha_{i_{j}}-\sum_{l=1}^{j}m_{jl}\varpi_{i_{j}}.$
			\end{itemize}
		\end{corollary}

		\begin{lemma}\label{lem4.3}
		 $\omega_{\widetilde{Z_{w}}}^{-1}$ is isomorphic to
			$\mathcal{O}_{\widetilde{Z_{w}}}(\alpha_{i_{r}})\otimes \mathcal{O}_{\widetilde{Z_{w[1]}}}(\alpha_{i_{r-1}})\otimes\cdots \otimes\mathcal{O}_{\widetilde{Z_{w[r-1]}}}(\alpha_{i_{1}})\otimes \widetilde{\pi_{w[r]}}^{*}\mathcal{L}(2\rho).$
		\end{lemma}
		\begin{proof}
			Consider the fibre product diagram
			
			$\centerline{\xymatrixcolsep{5pc}\xymatrix{ \widetilde{Z_{w}}\ar[r]\ar[d]_{\widetilde{\pi_{w[1]}}}& G/B\times G/B\ar[d]^{\widetilde{\pi_{r}}}\\
					\widetilde{Z_{w[r-1]}}	\ar[r]& G/B\times G/P_{\alpha_{i_{r}}}}}.$
			
			Then we have the following short exact sequence  
			\begin{equation}\label{eq4.1}
				0\longrightarrow \mathcal{R}_{\widetilde{\pi{w[1]}}}\longrightarrow T_{\widetilde{Z_{w}}}\longrightarrow \widetilde{\pi_{w[1]}}^{*} T_{\widetilde{Z_{w[1]}}}\longrightarrow 0,
			\end{equation}
		of tangent bundles on $\widetilde{Z_{w}},$ where $T_{\widetilde{Z_{w}}}$ (respectively, $T_{\widetilde{Z_{w[1]}}}$) denotes the tangent bundle of $\widetilde{Z_{w}}$ (respectively, of $\widetilde{Z_{w[1]}}$) and $\mathcal{R}_{\widetilde{ \pi_{w[1]}}}$ denotes the relative tangent bundle with respect to the map $\widetilde{\pi_{w[1]}}.$ Hence, by \eqref{eq4.1} we have $\omega_{\widetilde{Z_{w}}}^{-1}=\mathcal{R}_{\widetilde{ \pi_{w[1]}}}\otimes \widetilde{\pi_{w[1]}}^{*}K_{\widetilde{Z_{w[1]}}}^{-1}.$
		Since $\mathcal{L}(\varepsilon_{0}\boxtimes \alpha_{i_{r}})$ is the relative tangent bundle on $G/B\times G/B$  with respect to $\widetilde{\pi_{r}},$ the pull back line bundle $\mathcal{O}_{\widetilde{Z_{w}}}(\alpha_{i_{r}})$ is equal to $\mathcal{R}_{\widetilde{ \pi_{w[1]}}}.$
			
			By induction we have $$\omega_{\widetilde{Z_{w[1]}}}^{-1}=\mathcal{O}_{\widetilde{Z_{w[1]}}}(\alpha_{i_{r-1}})\otimes \mathcal{O}_{\widetilde{Z_{w[2]}}}(\alpha_{i_{r-2}})\otimes\cdots\otimes \mathcal{O}_{\widetilde{Z_{w[r-2]}}}(\alpha_{i_{1}})\otimes \widetilde{\pi_{w[r-1]}}^{*}\mathcal{L}(2\rho).$$ Therefore, we have $$\omega_{\widetilde{Z_{w}}}^{-1}=\mathcal{O}_{\widetilde{Z_{w}}}(\alpha_{i_{r}})\otimes \mathcal{O}_{\widetilde{Z_{w[1]}}}(\alpha_{i_{r-1}})\otimes \mathcal{O}_{\widetilde{Z_{w[2]}}}(\alpha_{i_{r-2}})\otimes\cdots\otimes \mathcal{O}_{\widetilde{Z_{w[r-1]}}}(\alpha_{i_{1}})\otimes \widetilde{\pi_{w[r]}}^{*}\mathcal{L}(2\rho).$$
			
		\end{proof}
 By \cref{lem4.1} we have $$\omega_{\widetilde{Z_{w}}}^{-1}=\mathcal{O}_{\widetilde{Z_{w}}}(n_{1},n_{2},\ldots,n_{r})\otimes \widetilde{\pi_{w[r]}}^{*}\mathcal{L}(2\rho+\mu)$$ for a unique tuple $(n_{1},\ldots, n_{r})\in \mathbb{Z}^{r}$ and a unique $\mu\in X(B).$ 
		\begin{proposition}\label{prop4.5}
		 Then we have
			\begin{itemize}
				\item[(i)] $n_{r}=2.$
				
				\item [(ii)] $n_{j}=\sum_{l=j}^{r}m_{lj}$ for all $1\le j\le r-1.$
				
				\item [(iii)] $\mu=\sum_{j=1}^{r}\lambda_{j}.$
			\end{itemize} 
		\end{proposition}
		\begin{proof}
			By \cref{lem4.3} we have $$\omega_{\widetilde{Z_{w}}}^{-1}=\mathcal{O}_{\widetilde{Z_{w[r-1]}}}(\alpha_{i_{1}})\otimes \cdots \otimes\mathcal{O}_{\widetilde{Z_{w[1]}}}(\alpha_{i_{r-1}})\otimes \mathcal{O}_{\widetilde{Z_{w}}}(\alpha_{i_{r}})\otimes \widetilde{\pi_{w[r]}}^{*}\mathcal{L}(2\rho).$$ 
			
			For $1\le j\le r,$ by using \cref{cor4.2} we have $$\mathcal{O}_{\widetilde{Z_{w[r-j]}}}(\alpha_{i_{j}})=\mathcal{O}_{\widetilde{Z_{w[r-j]}}}(m_{j1},m_{j2},\ldots,m_{jj})\otimes \widetilde{\pi_{w[r-j][j]}}^{*}\mathcal{L}(\lambda_{j}),$$ where $m_{jj}=2,$ $m_{jk}=\langle \alpha_{i_{j}}-\sum_{l=k+1}^{j}m_{jl}\varpi_{i_{l}},\alpha_{i_{k}}^{\vee}\rangle$ for all $1\le k\le j-1$ and $\lambda_{j}=\alpha_{i_{j}}-\sum_{l=1}^{j}m_{jl}\varpi_{\alpha_{i_{j}}}.$ Therefore, we have  $$\omega_{\widetilde{Z_{w}}}^{-1}=\mathcal{O}_{\widetilde{Z_{w}}}(\sum_{l=1}^{r}m_{l1},\ldots,\sum_{l=j}^{r}m_{lj},\ldots,m_{rr})\otimes\widetilde{\pi_{w[r]}}^{*}\mathcal{L}(\sum_{j=1}^{r}\lambda_{j}).$$ Since $\{\mathcal{O}_{\widetilde{Z_{w[j]}}}(1): 1\le j\le r\}\cup\{\widetilde{\pi_{w[r]}}^{*}\mathcal{L}(\varpi_{j}): 1\le j\le n\}$ forms a basis of \text{Pic}($\widetilde{Z_{w}}$), we have  $n_{r}=m_{rr}=2,$ $n_{j}=\sum_{l=j}^{r}m_{lj}$ ($1\le j\le r-1$) and $\mu=\sum_{j=1}^{r}\lambda_{r}.$
			
		\end{proof}
		\begin{theorem}\label{thm5.6}
		$\omega_{\widetilde{Z_{w}}}^{-1}$ is globally generated (respectively, ample) if and only if 
			\begin{itemize}
				\item [(i)] $\sum_{l=j}^{r}m_{lj}\ge 0 ~(\text{respectively,} \sum_{l=j}^{r}m_{lj} >0)$ for all $1\le j\le r-1$ and
				
				\item [(ii)] $2\rho+\sum_{j=1}^{r}\lambda_{j}$ is dominant (respectively, regular dominant).
			\end{itemize} 
		\end{theorem}
		\begin{proof}
			By using \cref{prop4.5}, \cref{thm3.1} and \cref{cor3.3} proof of the theorem follows.
		\end{proof}

		\begin{proof}[Proof of \cref{thm5.7}]
			By \cref{thm3.5} and \cref{thm4.8}   $\omega_{\widetilde{Z_{w}}}^{-1}$ is weak-Fano if and only if $\omega_{\widetilde{Z_{w}}}^{-1}$ is globally generated. Therefore, proof follows from \cref{thm5.6}. 
		\end{proof}
		
		\begin{rem}
		Let $P$ ($\supseteq B$) be a parabolic subgroup of $G.$ For a sequence $w$ of simple reflections in $W$, consider the variety $Y_{w, P}=P\times^{B} Z_{w}.$ In particular, for $P=B$ (respectively, $P=G$),  $Y_{w,P}$ becomes $Z_{w}$ (respectively, $\widetilde{Z_{w}}$). Then all the results proved in our article for $\widetilde{Z_{w}}$ can be carried over for $Y_{w,P}$ with some appropriate modifications.
		\end{rem}
		
		\begin{rem}
			Even if the anti-canonical line bundle of $Z_{w}$ is very ample (respectively, globally generated), the anti-canonical line bundle of $\widetilde{Z_{w}}$ may not be very ample (respectively, globally generated). The following example illustrates this fact. 
		\end{rem}

		\begin{exam}\label{exam5.9}
			If $w=(s_{\alpha})$ for some $\alpha\in S$, then $\widetilde{Z_{w}}=G\times^{B}P_{\alpha}/B.$ 
			Recall that we have the natural projection map $\widetilde{\pi_{w[1]}}:\widetilde{Z_{w}}\to G/B.$ Thus we have the following $B$-equivariant fibre product diagram:
			
			$$\centerline{\xymatrixcolsep{5pc}\xymatrix{ P_{\alpha}/B\ar[r]^{\varphi_{w}}\ar[d]_{\pi_{w[1]}}& G/B\ar[d]^{\pi_{\alpha}}\\
					B/B\ar@/^2pc/[u]^{\sigma_{w(1)}}	\ar[r]^{\varphi_{w[1]}}& G/P_{\alpha}}}$$
				where $\sigma_{w(1)}$ is the natural $B$-equivariant section of $\pi_{w[1]}.$
			
			This induces the $G$-equivariant fibre product diagram:
			
			$$\centerline{\xymatrixcolsep{5pc}\xymatrix{\widetilde{Z_{w}}\ar[r]^{\widetilde{\varphi_{w}}}\ar[d]_{\widetilde{\pi_{w[1]}}}& G/B\times G/B\ar[d]^{\widetilde{\pi_{\alpha}}}\\
					G\times^{B} B/B \ar[r]_{\widetilde{\varphi_{w[1]}}}\ar[ru]^{\Delta} \ar@/^3pc/[u]^{\widetilde{\sigma_{w(1)}}} & G/B \times G/P_{\alpha}}}$$
		where $\widetilde{\sigma_{w(1)}}$ is the natural $G$-equivariant section of $\widetilde{\pi_{w[1]}}.$ Note that there is a $G$-equivariant isomorphism between $G\times^{B} B/B$ and $G/B.$ We identify $G\times^{B} B/B$ with $G/B$ via this isomorphism.
		
   Consider the following short exact sequence 
			\begin{equation}\label{eq5.2}
				0\to \mathcal{R}_{\widetilde{ \pi_{w[1]}}}\to T_{\widetilde{Z_{w}}}\to \widetilde{\pi_{w[1]}}^{*}T_{G/B}\to 0
			\end{equation}
of tangent bundles on $\widetilde{Z_{w}},$ where $\mathcal{R}_{\widetilde{ \pi_{w[1]}}}$ is the relative tangent bundle with respect to the map $\widetilde{\pi_{w[1]}}.$ 

Since $\omega_{G/B}^{-1}=\mathcal{L}(2\rho),$ by using \eqref{eq5.2} we have $\omega_{\widetilde{Z_{w}}}^{-1}=\mathcal{R}_{\widetilde{ \pi_{w[1]}}}\otimes \widetilde{\pi_{w[1]}}^{*}\mathcal{L}(2\rho).$ 
Since $\widetilde{ \pi_{w[1]}}$ is $P_{\alpha}/B(\simeq \mathbb{P}^{1})$-fibration, we have $$\mathcal{R}_{\widetilde{ \pi_{w[1]}}}=\mathcal{O}_{\widetilde{Z_{w}}}(m)\otimes \widetilde{\pi_{w[1]}}^{*}\mathcal{L}(\mu)$$ for a unique $m\in \mathbb{Z}$ and a unique $\mu\in X(B),$ where $\mathcal{O}_{\widetilde{Z_{w}}}(1):=\mathcal{O}_{\widetilde{Z_{w}}}(\varpi_{\alpha}).$ 

Since the restriction of $\mathcal{R}_{\widetilde{ \pi_{w[1]}}}$ to the fibre over $B/B$ is equal to $\omega_{\mathbb{P}^{1}}^{-1},$ we have $m=2.$ By the above diagram we have $\widetilde{\varphi_{w}}\widetilde{\sigma_{w(1)}}=\Delta$ (but, note that $\Delta\widetilde{ \pi_{w[1]}}\neq \widetilde{\varphi_{w}}$). Thus, we have $\widetilde{\sigma_{w(1)}}^{*}\mathcal{R}_{\widetilde{ \pi_{w[1]}}}=\Delta^{*}(\mathcal{L}(\varepsilon_{0}\boxtimes \alpha))=\mathcal{L}(\alpha).$  Again, we have $\widetilde{\sigma_{w(1)}}^{*}\mathcal{O}_{\widetilde{Z_{w}}}(1)=\mathcal{L}(\varpi_{\alpha}).$ Therefore, we have $\mathcal{R}_{\widetilde{ \pi_{w[1]}}}=\mathcal{O}_{\widetilde{Z_{w}}}(2)\otimes \widetilde{\pi_{w[1]}}^{*}\mathcal{L}(\alpha-2\varpi_{\alpha}).$

Hence, by \cref{thm5.6} $\omega_{\widetilde{Z_{w}}}^{-1}$ is very ample (respectively, globally generated) if and only if $2\rho+\alpha-2\varpi_{\alpha}$ is regular dominant (respectively, dominant). 

If $G$ is simply-laced, then $2\rho+\alpha-2\varpi_{\alpha}$ is regular dominant, hence $\widetilde{Z_{w}}$ is Fano. On the other hand, if $G$ is not type $G_{2},$ then $2\rho+\alpha-2\varpi_{\alpha}$ is always dominant. So, $\omega_{\widetilde{Z_{w}}}^{-1}$ is always globally generated, hence weak-Fano (cf. \cref{cor4.10}).

When $G$ is of type $G_{2},$ let $\alpha_{1},\alpha_{2}$ be two simple roots of $G$ such that $\langle \alpha_{2},\alpha_{1}^{\vee} \rangle=-3.$ Now, if  $\alpha=\alpha_{1},$  then $2\rho+\alpha-2\varpi_{\alpha}$ is regular dominant, hence $\widetilde{Z_{w}}$ is Fano. On the other hand, if $\alpha=\alpha_{2},$ then $2\rho+\alpha-2\varpi_{\alpha}$ is not dominant, hence $\widetilde{Z_{w}}$ is not weak-Fano. 
\end{exam}

		\section{Equivariant vector bundle on $G$-BSDH-variety}\label{sec6}
	
	Let $w=(s_{i_{1}},s_{i_{2}},\ldots, s_{i_{r}})$ be a reduced sequence of simple reflections in $W.$  In this section we prove that for a $T$-equivariant vector bundle $\mathcal{E}$ on $\widetilde{Z_{w}},$ $\mathcal{E}$ is nef (respectively, ample) if and only if the restriction $\mathcal{E}|_{C}$ of $\mathcal{E}$ to every $T$-invariant curve $C$ on $\widetilde{Z_{w}}$ is nef (respectively, ample). Recall that by a curve on $\widetilde{Z_{w}}$ we mean a closed irreducible one dimensional subvariety of $\widetilde{Z_{w}}.$

	 For each $\alpha_{i},$ let $P_{S\setminus\{\alpha_{i}\}}$ denote the unique maximal parabolic subgroup of $G$ containing $B.$  By  \cite[Theorem 1(ii), p.608]{Mag1} we have $$Z_{w}=\overline{B\cdot z_{w}}\hookrightarrow \prod_{j=1}^{r}G/P_{S\setminus\{\alpha_{i_{j}}\}}$$ where $z_{w}= (s_{i_{1}}P_{S\setminus\{\alpha_{1}\}}, s_{i_{1}}s_{i_{2}}P_{S\setminus\{\alpha_{i_{2}}\}} , \ldots, s_{i_{1}}\cdots s_{i_{r}} P_{S\setminus\{\alpha_{i_{r}}\}})\in \prod_{j=1}^{r}G/P_{S\setminus\{\alpha_{i_{j}}\}}$ and $B$-acts diagonally on $\prod_{j=1}^{r}G/P_{S\setminus\{\alpha_{i_{j}}\}}.$ Then 
		\begin{lemma}\label{lem6.1}
			The restriction map $${\rm res}: {\rm Pic}(\prod_{j=1}^{r}G/P_{S\setminus\{\alpha_{i_{j}}\}}) \longrightarrow {\rm Pic}(Z_{w})$$ of Picard groups induced by the above embedding is an isomorphism.
		\end{lemma}
		\begin{proof}
			See \cite[Section 3.1, p.464]{LT}.
		\end{proof}
	
		Since $Z_{w}$ embeds $B$-equivariantly inside $\prod_{j=1}^{r}G/P_{S\setminus\{\alpha_{i_{j}}\}},$ we have a $G$-equivariant embedding  $$\widetilde{Z_{w}}=G\times^{B}\overline{B\cdot z_{w}}\hookrightarrow G\times^{B} \prod_{j=1}^{r}G/P_{S\setminus\{\alpha_{i_{j}}\}}\xrightarrow{\sim} G/B\times \prod_{j=1}^{r}G/P_{S\setminus\{\alpha_{i_{j}}\}}$$ 
		where the last isomorphism is given by the map $$[g, (g_{1}P_{S\setminus\{\alpha_{i_{1}}\}},\ldots, g_{r}P_{S\setminus\{\alpha_{i_{r}}\}})]\mapsto (gB, (gg_{1}P_{S\setminus\{\alpha_{i_{1}}\}},\ldots, gg_{r}P_{S\setminus\{\alpha_{i_{r}}\}}))$$
		and the action of $G$ on $G/B\times \prod_{j=1}^{r}G/P_{S\setminus\{\alpha_{i_{j}}\}}$ is given by diagonal action. Then
		
		\begin{lemma}\label{lem6.2}
			The restriction map $${\rm res }: {\rm Pic}(G \times^{B} \prod_{j=1}^{r}G/P_{S\setminus\{\alpha_{i_{j}}\}}) \longrightarrow {\rm Pic}(\widetilde{Z_{w}})$$ of Picard groups induced by the above embedding is an isomorphism.
		\end{lemma}
		\begin{proof}
		Using Kempf vanishing theorem it follows that $H^{1}(G/P,\mathcal{O}_{G/P})=0$ for any parabolic subgroup $P$ of $G.$ Thus, by \cite[Chapter III, Exercicise 12.6(b), p.292]{Har} we have ${\rm Pic}(G\times^{B} \prod_{j=1}^{r}G/P_{S\setminus\{\alpha_{i_{j}}\}})={\rm Pic}(G/B)\times \prod_{j=1}^{r} {\rm Pic}(G/P_{S\setminus\{\alpha_{i_{j}}\}}).$ Therefore, Picard group of $$G\times^{B} \prod_{j=1}^{r} {\rm Pic}(G/P_{S\setminus\{\alpha_{i_{j}}\}})$$ is a free abelian group of rank $n+r,$ where $n$ is the rank of $G.$

		Consider the following $G$-equivariant commutative triangle
		
	$\centerline{\xymatrix{
			\widetilde{Z_{w}} \ar[d] \ar@{^{(}->}[r]
			&  G\times^{B} \prod_{j=1}^{r}G/P_{S\setminus\{\alpha_{i_{j}}\}}\ar[dl]\\
			G/B&}}.$
		
 The way we have constructed the $\mathcal{O}(1)$-basis of ${\rm Pic}(\widetilde{Z_{w}})$ implies that the restriction map $${\rm res }: {\rm Pic}(G \times^{B} \prod_{j=1}^{r}G/P_{S\setminus\{\alpha_{i_{j}}\}}) \longrightarrow {\rm Pic}(\widetilde{Z_{w}})$$ of Picard groups induced by the above embedding is a surjective. Using \cref{lem4.1} it follows that ${\rm Pic}(\widetilde{Z_{w}})$ is a free abelian group of rank $n+r.$ Therefore, the restriction map $${\rm res }: {\rm Pic}(G \times^{B} \prod_{j=1}^{r}G/P_{S\setminus\{\alpha_{i_{j}}\}}) \longrightarrow {\rm Pic}(\widetilde{Z_{w}})$$ of Picard groups is an isomorphism.
		
		\end{proof}

		\begin{proof}[Proof of \cref{thm6.3}]
			If $\mathcal{E}$ is a $T$-equivariant nef (respectively, ample) vector bundle on $\widetilde{Z_{w}},$ then the restriction $\mathcal{E}|_{C}$ of $\mathcal{E}$ to every $T$-invariant curve $C$ of $\widetilde{Z_{w}}$ is nef (respectively, ample).
			
			Conversely, if $\mathcal{E}$ is a $T$-equivariant vector bundle on $\widetilde{Z_{w}}$ such that the restriction $\mathcal{E}|_{C}$ of $\mathcal{E}$ to every $T$-invariant closed curve $C$ on $\widetilde{Z_{w}}$ is nef, then by using the arguments similar to the proof of \cite[Theorem 3.1, p.6]{BHK} it follows that $\mathcal{E}$ is nef.
			
			Assume that $\mathcal{E}$ is a $T$-equivariant vector bundle on $\widetilde{Z_{w}}$ such that the restriction $\mathcal{E}|_{C}$ of $\mathcal{E}$ to every $T$-invariant curve $C$ on $\widetilde{Z_{w}}$ is ample.

			Now by \cref{lem6.2} it follows that the restriction map  $${\rm res }: {\rm Pic}(G/B\times\prod_{j=1}^{r}G/P_{S\setminus\{\alpha_{i_{j}}\}}) \longrightarrow {\rm Pic}(\widetilde{Z_{w}})$$ of Picard groups is an isomorphism.
			
			Then by using the argument similar to the argument used in the proof \cite[Theorem 3.1, p.6]{BHK} it follows that $\mathcal{E}$ is ample.  
			
		\end{proof}
\begin{rem}
Let $X$ be a projective $T$-variety. Using the arguments similar to the proof of \cite[Theorem 3.1, p.6]{BHK} it follows that  a $T$-equivariant vector bundle $\mathcal{E}$ on $X$ is nef  if and only if the restriction $\mathcal{E}|_{C}$ of $\mathcal{E}$ to every $T$-invariant  curve $C$ on $X$ is nef. 
\end{rem}

		\subsection*{Acknowledgement}
		The second named author would like to thank the
		Indian Institute of Technology Bombay for the postdoctoral fellowship and the hospitality during his stay.

	\end{document}